\documentclass[final,onefignum,onetabnum]{siamonline220329}

\usepackage{braket,amsfonts}

\usepackage{array}
\usepackage[caption=false]{subfig}
\usepackage{pgfplots}

\newsiamthm{claim}{Claim}
\newsiamremark{remark}{Remark}
\newsiamremark{hypothesis}{Hypothesis}
\crefname{hypothesis}{Hypothesis}{Hypotheses}

\usepackage{algorithmic}

\usepackage{graphicx,epstopdf}
\usepackage{subfig}
\usepackage{xcolor}
\definecolor{matlab_orange}{rgb}{1.00,0.41,0.16}
\definecolor{matlab_blue}{rgb}{0.00,0.45,0.74}
\definecolor{matlab_purple}{rgb}{0.49,0.18,0.56}

\Crefname{ALC@unique}{Line}{Lines}

\usepackage{amsopn}
\usepackage{amssymb}

\usepackage{xspace}
\usepackage{bold-extra}
\usepackage[most]{tcolorbox}

\colorlet{texcscolor}{blue!50!black}
\colorlet{texemcolor}{red!70!black}
\colorlet{texpreamble}{red!70!black}
\colorlet{codebackground}{black!25!white!25}

\usepackage[mathlines]{lineno}

\usepackage{fixmath}

\lstdefinestyle{siamlatex}{%
  style=tcblatex,
  texcsstyle=*\color{texcscolor},
  texcsstyle=[2]\color{texemcolor},
  keywordstyle=[2]\color{texemcolor},
  moretexcs={cref,Cref,maketitle,mathcal,text,headers,email,url},
}

\tcbset{%
  colframe=black!75!white!75,
  coltitle=white,
  colback=codebackground, 
  colbacklower=white, 
  fonttitle=\bfseries,
  arc=0pt,outer arc=0pt,
  top=1pt,bottom=1pt,left=1mm,right=1mm,middle=1mm,boxsep=1mm,
  leftrule=0.3mm,rightrule=0.3mm,toprule=0.3mm,bottomrule=0.3mm,
  listing options={style=siamlatex}
}

\newtcblisting[use counter=example]{example}[2][]{%
  title={Example~\thetcbcounter: #2},#1}

\newtcbinputlisting[use counter=example]{\examplefile}[3][]{%
  title={Example~\thetcbcounter: #2},listing file={#3},#1}

\DeclareTotalTCBox{\code}{ v O{} }
{ 
  fontupper=\ttfamily\color{black},
  nobeforeafter,
  tcbox raise base,
  colback=codebackground,colframe=white,
  top=0pt,bottom=0pt,left=0mm,right=0mm,
  leftrule=0pt,rightrule=0pt,toprule=0mm,bottomrule=0mm,
  boxsep=0.5mm,
  #2}{#1}

\patchcmd\newpage{\vfil}{}{}{}
\begin{tcbverbatimwrite}{tmp_\jobname_header.tex}
\title{Noninvasive Adaptive Control of a Class of Nonlinear Systems With Unknown Parameters\thanks{Submitted to the editors DATE.
\funding{This work was supported by the Engineering and Physical Sciences Research Council (EP/W032236/1).}}}

\author{Hamed Rezaee\thanks{Department of Mechanical Engineering,  Imperial College London, London, UK (\email{h.rezaee@imperial.ac.uk}).}
	\and Ludovic Renson\thanks{Department of Mechanical Engineering,  Imperial College London, London, UK (\email{l.renson@imperial.ac.uk}).}}

\headers{Noninvasive Adaptive Control of a Class of Nonlinear Systems }{Rezaee and Renson}\end{tcbverbatimwrite}
\title{Noninvasive Adaptive Control of a Class of Nonlinear Systems With Unknown Parameters\thanks{Submitted to the editors DATE.
\funding{This work was supported by the Engineering and Physical Sciences Research Council (EP/W032236/1).}}}

\author{Hamed Rezaee\thanks{Department of Mechanical Engineering,  Imperial College London, London, UK (\email{h.rezaee@imperial.ac.uk}).}
	\and Ludovic Renson\thanks{Department of Mechanical Engineering,  Imperial College London, London, UK (\email{l.renson@imperial.ac.uk}).}}

\headers{Noninvasive Adaptive Control of a Class of Nonlinear Systems }{Rezaee and Renson}

\ifpdf
\hypersetup{ pdftitle={Guide to Using  SIAM'S \LaTeX\ Style} }
\fi

\begin{document}
\maketitle

\begin{tcbverbatimwrite}{tmp_\jobname_abstract.tex}
\begin{abstract}
Control-based continuation (CBC) is a general and systematic method to explore the dynamic response of a physical system and perform bifurcation analysis directly during experimental tests. Although CBC has been successfully demonstrated on a wide range of systems, rigorous and general approaches to designing a noninvasive controller underpinning the methodology are still lacking. In this paper, a noninvasive adaptive control strategy for a wide class of nonlinear systems with unknown parameters is proposed. We prove that the proposed adaptive control methodology is such that the states of the dynamical system track a reference signal in a noninvasive manner if and only if the reference is a response of the uncontrolled system to  an excitation force. Compared to the existing literature, the proposed method does not require any a priori knowledge of some system parameters, does not require a persistent excitation, and is not restricted to linearly-stable systems, facilitating the application of CBC to a much larger class of systems than before. Rigorous mathematical analyses are provided, and the proposed control method is numerically demonstrated on a range of single- and multi-degree-of-freedom nonlinear systems, including a Duffing oscillator with multiple static equilibria. It is especifically shown that the unstable periodic orbits of the uncontrolled systems can be stabilized and reached, noninvasively, in controlled conditions.

\end{abstract}

\vspace{0.1 in}
\end{tcbverbatimwrite}
\begin{abstract}
Control-based continuation (CBC) is a general and systematic method to explore the dynamic response of a physical system and perform bifurcation analysis directly during experimental tests. Although CBC has been successfully demonstrated on a wide range of systems, rigorous and general approaches to designing a noninvasive controller underpinning the methodology are still lacking. In this paper, a noninvasive adaptive control strategy for a wide class of nonlinear systems with unknown parameters is proposed. We prove that the proposed adaptive control methodology is such that the states of the dynamical system track a reference signal in a noninvasive manner if and only if the reference is a response of the uncontrolled system to  an excitation force. Compared to the existing literature, the proposed method does not require any a priori knowledge of some system parameters, does not require a persistent excitation, and is not restricted to linearly-stable systems, facilitating the application of CBC to a much larger class of systems than before. Rigorous mathematical analyses are provided, and the proposed control method is numerically demonstrated on a range of single- and multi-degree-of-freedom nonlinear systems, including a Duffing oscillator with multiple static equilibria. It is especifically shown that the unstable periodic orbits of the uncontrolled systems can be stabilized and reached, noninvasively, in controlled conditions.

\end{abstract}

\vspace{0.1 in}


\section{Introduction}
The general objective of a control system is to use the inputs of a dynamical system to lead state variables toward a desired value or trajectory. Noninvasive control refers to a particular class of control systems for which the control inputs asymptotically become zero if the states of the system approach responses of the underlying uncontrolled system. Despite vanishing, the controller can provide a stabilizing effect such that desired responses can be observed even if they correspond to unstable periodic orbits of the underlying uncontrolled system. An early application of this idea is chaos control where unstable periodic orbits embedded in a chaotic attractor are stabilized~\cite{ChaosControlHandbook,Boccaletti00}. In this context, numerous noninvasive control approaches have been developed such as the so-called OGY method~\cite{OGY,Rega10}, time delayed feedback (TDF)~\cite{Pyragas92,FlunkertPhysRevE:11} and active filters~\cite{PyraPhysLett:07,LuPhyLet:08}. Chaos control has found applications in electronic circuits~\cite{PyraPhysLett:07,LuPhyLet:08}, chemical oscillators~\cite{Peng91,Petrov93}, lasers and nonlinear optics~\cite{BenklerJOpt:00,SchikoraPhysRevE:08}, secure communication~\cite{JustChaos:23,Hayes94}, and fluid, mechanical, biological, and biochemical systems~\cite{Singer91,Garfinkel92}. More recent applications of noninvasive control use phase-locked loops (PLLs) to stabilize periodic oscillations in a range of mechanical systems \cite{MojrzischPAMM:12,HippoldArxiv:24,Abeloos22,DenisMSSP:18,PeterMSSP:17}.

More generally, noninvasive control provides a systematic means to explore the dynamics of a system purely based on its input-output relationship. The controller can be designed to steer the controlled dynamics toward responses described by a so-called reference signal. The noninvasive control inputs asymptotically vanish if the reference signal corresponds to a response of the uncontrolled system, which guarantees that the response is observed without affecting its position in state and parameter spaces. Finding a reference signal that results in vanishing control inputs defines a zero-problem whose solutions can be found iteratively and tracked in parameter space using suitable path-following techniques~\cite{Schilder2015}. Control-based continuation (CBC)~\cite{Sieber2008} exploits this idea for bifurcation analysis of physical systems directly during experiments. If properly designed, the control system can reduce (or even overcome) the sensitivity to initial conditions and stability changes due to bifurcations, which results in more repeatable experimental tests compared to traditional (uncontrolled) approaches. CBC has been used on a wide range of mechanical structures to measure important dynamics characteristics such as frequency response curves~\cite{BartonJVC:12,BartonPhysRevE:13,Bureau13,Schilder2015,KleymanMRC:20}, backbone curves~\cite{Renson2016,Abeloos22}, and fold bifurcation curves~\cite{Renson19b,Renson17}. More recently, CBC was used to characterize experimentally the equilibria of the Zeeman catastrophe machine~\cite{Dittus23}, the dynamics of bubbles~\cite{Fontana22}, the buckling of beams~\cite{Neville18,Shen21a,Shen21b}, and the limit cycle oscillations of an airfoil~\cite{Lee23,BeregiNonDyn:21}, and demonstrated numerically on equilibria and periodic orbits in biological systems~\cite{deCesare22,BlythNonDyn:23}. Identified features can then be exploited for model development, calibration, and validation~\cite{Lee23,Beregi23,Song23}.
CBC contrasts with other noninvasive control methods like the OGY, TDF, and PLL in that it uses a reference signal. The reference signal acts as a proxy for the system states, which facilitates steering the system dynamics toward desired behaviors and exploring regions where multiple responses coexist. However, solving for the reference signal to achieve noninvasive control can be time-consuming, especially for slow-fast systems as investigated in~\cite{BlythNonDyn:23}.

Although CBC has been exploited in numerous studies, the majority of the existing literature primarily emphasizes its application rather than the systematic design and theoretical analysis of its underlying noninvasive control method. As such, new applications often require a significant number of trials and errors to obtain a controller that works throughout the parameter range of interest. In this paper, we contribute to filling this void by proposing a systematic noninvasive control strategy. Based on the idea of adaptive control, the gain/parameters of the control system are not fixed. They will be adjusted under an adaptation strategy that guarantees the tracking of desired reference signals even if the parameters of the systems are unknown. The proposed control strategy addresses a wide class of high-order nonlinear systems with multiple degrees-of-freedom, and is designed such that the desired reference signal can be tracked in a noninvasive manner if and only if it is a response of the underlying uncontrolled system.

Recent work in this direction was presented in~\cite{PanagiotopoulosSIAM:23} where the design of noninvasive control using various techniques such as zero-in-equilibrium feedback control and washout filters was addressed. However, due to model nonlinearities and parameter uncertainties, achieving a desired tracking performance based on linear approaches may not be guaranteed. Accounting for model uncertainty, a noninvasive control method based on model-reference adaptive control and pole-placement adaptive control was proposed for linear discrete-time systems with periodic behavior in~\cite{LiJVC:20}. The approach was then extended to linear continuous-time systems with periodic behavior in~\cite{LiNonDyn:21} and to a class of nonlinear systems with periodic behavior in~\cite{LiNonDyn:22}. In~\cite{LiNonDyn:21} and \cite{LiNonDyn:22}, the linear terms of the system were assumed to be Hurwitz stable and known. This helps design a model reference for a proper adaptive controller and guarantees stability. However, this assumption prevents the use of the controller in applications exhibiting multiple static equilibria (such as buckled beams and plates) and systems exhibiting self-excited oscillations (as in machine tool vibration, aeroelastic flutter, etc.). Another assumption made in~\cite{LiNonDyn:21} and \cite{LiNonDyn:22} is that the desired response is persistently exciting. Based on such an assumption, the control input associated with a desired response should be rich enough (in terms of harmonic modes) to fully excite the dynamics of the system. Accordingly, it becomes possible to uniquely estimate the true unknown parameters of the system. This knowledge of the true parameters is then used in the controller to guarantee asymptotic tracking of the desired response together with the noninvasiveness of the control input. As shown in Section~\ref{SecSimulations}, satisfying this assumption is not always possible or may require a (very) fine discretization of the dynamics (using, for instance, a large number of Fourier modes), which increases the complexity and cost of solving the CBC root problem.

The control strategy proposed in the present paper lifts those assumptions such that no knowledge of any model parameter is needed, no stability assumptions are made on the linear terms of the system, and the performance of the proposed control strategy does not rely on the persistent excitation of the desired response. This makes our control method applicable to a much larger range of systems than before. 

The structure of this paper is as follows. Notation is provided in the next section. Section~\ref{SecProblem} defines the noninvasive control problem in more precise terms, and the proposed noninvasive adaptive control method is presented 
in Section~\ref{SecControlDesign} along with rigorous proofs of its convergence and stability. In Section~\ref{SecSimulations}, simulation results on three different mechanical structures demonstrate the performance of the proposed control strategy, and Section~\ref{SecConclusions} concludes this paper.

\section{Notation}\label{Sec:Notation}
Throughout the paper $\mathbb{R}$ and $\mathbb{R}_{>0}$ denote the set of real and positive real numbers, respectively. $\mathbb{R}^{n}$ and $\mathbb{R}^{n\times m}$ denote the sets of $n\times 1$ real vectors and $n\times m$ real matrices, respectively.
$v^{(k)}(t)$ denotes the $k^{\text{th}}$ derivative of a variable $v(t)$. $\mathbf{0}_{k}$ is a $k\times 1$ vector of zeros. $|\cdot|$ denotes the absolute value. $\|\cdot\|$ denotes the 2-norm. $|\cdot|_\infty$ denotes the supremum value, and ${\rm sat}(\cdot)$ denotes the saturation function, which for a scalar $v$, is defined as
\begin{equation*}\label{satfunc}
\mathrm{sat}(v/\epsilon)=\left\{\begin{array}{rl}
1&\quad v \geq \epsilon\\
v/{\epsilon}&\quad -\epsilon \leq v \leq \epsilon\\
-1&\quad v \leq -\epsilon,\\
\end{array}
\right.
\end{equation*}
where $\epsilon \in \mathbb{R}_{>0}$. Moreover,  $\mathrm{pol}(\lambda_i,k)$ is a polynomial of degree $k$ with parameters $\lambda_i\in \mathbb{R},i=1,2,\ldots,k,$ defined as $
\mathrm{pol}(\lambda_i,k)=s^{k}+\lambda_{k}s^{k-1}+\ldots+\lambda_2s+\lambda_1$.

\section{Problem Statement}\label{SecProblem}
Consider a class of nonlinear systems with $p$ degrees-of-freedom described as follows:
\begin{equation}\label{system}
\begin{split}
    \mathbold{x}^{(n)}(t)&=\mathbold{F}(\mathbold{\xi}(t))\mathbold{\theta}+\mathbold{u}(t)\\
 \mathbold{\xi}&=\begin{bmatrix}
                  \mathbold{x}^{(n-1)\top } (t) & \ldots &   \dot{\mathbold{x}}^\top(t)  &   \mathbold{x}^\top(t)
                 \end{bmatrix}^\top,
\end{split}
\end{equation}
where $\mathbold{x}^{(k)}(t)\in \mathbb{R}^p, k\in \{0,1,2,\ldots,n-1\},$ are the state vectors (assumed to be fully measurable), $\mathbold{F}(\mathbold{\xi}(t))\in \mathbb{R}^{p\times m}$ is a locally Lipschitz nonlinear function of the states which is known, $\mathbold{\theta}\in \mathbb{R}^m$ is a vector of unknown constant parameters, and $\mathbold{u}(t)\in \mathbb{R}^p$ is the vector of inputs. This system can model a wide range of nonlinear systems that are linear in parameters (e.g., forced Duffing oscillators~\cite{AbeloosNonDyn:21,LenciMSSP:22}, forced van der Pol oscillators~\cite{SIAMDyn:03,MaCSF:21}, aircraft attitude dynamics~\cite{RichardsonAJ:06}, and beam structures~\cite{RensonProcA:18,EhrhardtJSV:19}).
System~\eqref{system} considers identity input gains. The control method proposed here extends to systems with input gains in the form $\mathbold{B}\mathbold{u}(t)$ if $\mathbold{B}$ is invertible and provided that $\mathbold{B}$ can be identified a priori through tests. In such cases, the value of $\mathbold{u}(t)$ can be uniquely determined from  $\mathbold{B}\mathbold{u}(t)$.

Assume that the bounded vector $\mathbold{\xi}^*(t)=\begin{bmatrix}
                  \mathbold{x}^{*(n-1)\top }(t)  & \ldots &   \dot{\mathbold{x}}^{*\top}(t)  &   \mathbold{x}^{*\top}(t)
                 \end{bmatrix}^\top$ is a so-called \textit{natural response} of the uncontrolled system~\eqref{system} to a harmonic force $\mathbold{u}(t)=\mathbold{\sigma}(t)$.  Such natural responses also include equilibria and limit cycles when there is no excitation, i.e., when $\mathbold{\sigma}(t)=\mathbf{0}_p$.
                 In an uncontrolled setting, various experiments for $\mathbold{u}(t)=\mathbold{\sigma}(t)$ but with different initial states may lead to different natural responses and/or significantly different transient behaviors. In a controlled setting, it is possible to compensate for the effect of the initial conditions such that $\mathbold{\xi}(t)$ converges to a desired natural response $\mathbold{\xi}^*(t)$ in a `reasonable' time, and the control input simultaneously converges to zero, making the controller noninvasive.

Let $\mathbold{u}(t)$ be the combination of the external excitation $\mathbold{\sigma}(t)$ and a control input $\mathbold{u}'(t)$ as $\mathbold{u}(t)=\mathbold{\sigma}(t)+\mathbold{u}'(t)$.
	Since $\mathbold{\theta}$ is unknown, we develop an adaptive control strategy for $\mathbold{u}'(t)$. We define $\mathbold{\xi}^r(t)=\begin{bmatrix}
		\mathbold{x}^{r(n-1)\top }(t)  & \ldots &   \dot{\mathbold{x}}^{r\top}(t)  &   \mathbold{x}^{r\top}(t)
	\end{bmatrix}^\top$ as a so-called \textit{reference signal}. The adaptive control strategy will steer the system dynamics toward the trajectory defined by this reference signal. We also recall that $\mathbold{\xi}^*(t)$ is a natural response of the system to $\mathbold{\sigma}(t)$, implying that 
\begin{equation}\label{systemopenloop}
\mathbold{x}^{*(n)}(t)=\mathbold{F}(\mathbold{\xi}^*(t))\mathbold{\theta}+\mathbold{\sigma}(t).
\end{equation}
The objective is to design the adaptive controller such that if $\mathbold{\xi}^r(t)= \mathbold{\xi}^*(t)$, we have
\begin{equation}\label{obj1}
    \lim_{t\rightarrow \infty} (\mathbold{\xi}(t)-\mathbold{\xi}^*(t))=\mathbf{0}_{np},
\end{equation}
while the control input asymptotically converges to zero as follows:
\begin{equation*}\label{obj2}
    \lim_{t\rightarrow \infty} \mathbold{u}'(t) =\mathbf{0}_{p}.
\end{equation*}

In the context of CBC, a natural response $\mathbold{\xi}^*(t)$ is a priori unknown and the reference signal $\mathbold{\xi}^r(t)$ should be iterated until the control input vanishes, implying here that $\mathbold{\xi}^r(t)= \mathbold{\xi}^*(t)$. This defines a zero-problem whose solutions can be found using Newton-like methods and tracked in parameter space by using suitable path-following techniques~\cite{Schilder2015,Renson19b}. 

\section{Noninvasive Adaptive Tracking Strategy}\label{SecControlDesign}
The proposed noninvasive control strategy is presented in this section. The main ideas underpinning the design are first presented before providing a theorem and the associated proof.

While noninvasive adaptive control of linear and nonlinear systems has already been addressed in the literature, partial knowledge of the system parameters and the persistently exciting nature of the responses were assumed~\cite{LiJVC:20,LiNonDyn:21,LiNonDyn:22}. In this context, the main achievements of the control method proposed here are to consider all the system parameters as unknown and to lift the persistent excitation requirement. To this end, the main idea of the proposed control strategy is to use the adaptive parameters for defining an auxiliary state. This contrasts with the approach followed in~\cite{LiJVC:20,LiNonDyn:21,LiNonDyn:22} where adaptive parameters are directly used to define the control input. The proposed auxiliary state is designed in such a way that if it converges to a desired value, determined based on the desired natural response for the uncontrolled system, the real states of the system asymptotically converge to the desired response. Hence, our objective is to design a noninvasive control strategy such that the auxiliary state converges to its desired value, whereas the adaptive parameters may not converge to the real parameters of the system. Since the control strategy does not rely on an exact estimation of the unknown parameters, the desired natural response is not required to be persistently exciting. 

The control input $\mathbold{u}'(t)$ is designed as
\begin{equation}\label{uprime}
\mathbold{u}'(t)=-k(\mathbold{z}(t)-\mathbold{x}^{r(n-1)}(0)) + \mathbold{\eta}(t),
\end{equation}
where $k\in \mathbb{R}_{>0}$. In the first part of the control input \eqref{uprime}, we introduce the auxiliary variable $\mathbold{z}(t)$, defined as
\begin{equation}\label{z}
   \mathbold{z}(t)=\mathbold{x}^{(n-1)}(t)-\int_{0}^{t}\Big(\mathbold{F}(\mathbold{\xi}(\tau))\hat{\mathbold{\theta}}(\tau)+\mathbold{\sigma}(\tau)+\mathbold{\eta}(\tau)\Big)d\tau.
\end{equation}
The definition of this auxiliary variable, is inspired by the natural dynamics $\mathbold{\xi}^*(t)$ of the system \eqref{system}, which satisfies \eqref{systemopenloop} and gives after integration
 \begin{equation}\label{xstarnm1}
	\mathbold{x}^{*(n-1)}(0)=\mathbold{x}^{*(n-1)}(t)-\int_{0}^{t}\Big(\mathbold{F}(\mathbold{\xi}^*(\tau))\mathbold{\theta} +\mathbold{\sigma}(\tau)\Big)d\tau.
\end{equation}
This part of the control input, which is proportional to the error between $\mathbold{z}(t)$ and $\mathbold{x}^{*(n-1)}(0)$, guarantees the convergence of $\mathbold{z}(t)$ to $\mathbold{x}^{*(n-1)}(0)$ such that the dynamics of the states $\mathbold{x}^{(n-1)}(t)$ follows that of $\mathbold{x}^{*(n-1)}(t)$.

Building upon the first part of the control input, the second part, $\mathbold{\eta}(t)$, should be designed such that $\mathbold{x}(t)$ and its derivatives converge to $\mathbold{x}^*(t)$ and its derivatives, respectively, and the objective \eqref{obj1} is achieved. More specifically, defining the tracking error $\mathbold{e}(t)=\mathbold{x}(t)-\mathbold{x}^r(t)$, we define $\mathbold{\eta}(t)$ as
 \begin{equation}\label{eta}
 	\mathbold{\eta}(t)=-\lambda_{n-1}\mathbold{e}^{(n-1)}(t)-\cdots-\lambda_1\dot{\mathbold{e}}(t)-\phi(t) \mathbold{y}(t) -g(\mathbold{\xi}(t),\mathbold{\xi}^r(t),\hat{\mathbold{\theta}}(t))\mathrm{sat}(\mathbold{y}(t)/\epsilon),
 \end{equation}
 where $\mathbold{y}(t)$ is
 \begin{equation}\label{y}
 	\mathbold{y}(t)= \mathbold{e}^{(n-1)}(t)+\lambda_{n-1}\mathbold{e}^{(n-2)}(t)+\cdots+\lambda_1\mathbold{e}(t).
 \end{equation}
Defining the vector $\mathbold{y}(t)$ allows us to control just one vector instead of controlling $n$ vectors $\mathbold{e}^{(k)}(t),\;k\in\{0,1,\ldots,n-1\}$, simplifying the design and analysis of the control strategy in the rest of the paper. In particular, by choosing the coefficients $\lambda_i\in \mathbb{R}_{>0},\;i\in \{1,2,\ldots,n-1\},$ such that the polynomial $\mathrm{pol}(\lambda_i,n-1)$ is Hurwitz stable, the convergence of $\mathbold{y}(t)$ to zero leads to the convergence to zero of $\mathbold{e}(t)$ and its derivatives, $\mathbold{e}^{(k)}(t),\;k\in\{1,2,\ldots,n-1\},$ as well. The terms $-\phi(t) \mathbold{y}(t)$ and $-g(\mathbold{\xi}(t),\mathbold{\xi}^r(t),\hat{\mathbold{\theta}}(t))\mathrm{sat}(\mathbold{y}(t)/\epsilon)$ in \eqref{eta}
   are stabilizing terms that make the control input robust to parameter uncertainties. More specifically, $\phi(t)\in \mathbb{R}$ is a time-varying parameter updated as follows:
 \begin{equation}\label{phidot}
 	\dot{\phi}(t)=\gamma \mathbold{y}^\top(t)\mathbold{y}(t),
 \end{equation}
 with $\gamma \in \mathbb{R}_{>0}$. Based on \eqref{phidot}, $\phi(t) $ is increasing (to be large enough for stabilizing) until $\mathbold{y}(t)$ becomes zero. A similar idea is considered in designing $g(\mathbold{\xi}(t),\mathbold{\xi}^r(t),\hat{\mathbold{\theta}}(t))$, which  is a control gain considered as follows:
 \begin{equation}\label{g}
 	g(\mathbold{\xi}(t),\mathbold{\xi}^r(t),\hat{\mathbold{\theta}}(t))=  \big\|(\mathbold{F}(\mathbold{\xi}(t))-\mathbold{F}(\mathbold{\xi}^r(t)))\hat{\mathbold{\theta}}(t)\big\| +\kappa,
 \end{equation}
 with $\kappa \in \mathbb{R}_{>0}$. According to \eqref{g}, a larger $\big\|(\mathbold{F}(\mathbold{\xi}(t))-\mathbold{F}(\mathbold{\xi}^r(t)))\hat{\mathbold{\theta}}(t)\big\| $ leads to larger gain $g(\mathbold{\xi}(t),\mathbold{\xi}^r(t),\hat{\mathbold{\theta}}(t))$ to compassionate the effect of uncertainties. According to \eqref{uprime}, if $\mathbold{z}(t)$ converges to $\mathbold{x}^{r(n-1)}(0)$ and the tracking error  $\mathbold{e}(t)$  and its derivatives converge to zero, $\mathbold{u}'(t)$ is noninvasive.  

Note that the `true system' parameters $\mathbold{\theta}$ cannot be used in $\mathbold{\eta}(t)$ and $\mathbold{z}(t)$ as they are unknown. Therefore, we use an estimate of $\mathbold{\theta}$, denoted by $\hat{\mathbold{\theta}}(t)$ and obtained as
\begin{equation}\label{thetahatdot}
 \dot{\hat{\mathbold{\theta}}}(t)=\mathbold{S}\mathbold{F}^\top(\mathbold{\xi}(t)) (\mathbold{z}(t)-\mathbold{x}^{r(n-1)}(0)),
\end{equation}
where $\mathbold{S}\in \mathbb{R}^{m\times m}$ is symmetric positive definite. According to \eqref{thetahatdot}, $\hat{\mathbold{\theta}}(t)$ will be updated until $\mathbold{z}(t)$ converges to $\mathbold{x}^{*(n-1)}(0)$.

\textit{Summary of the control parameters:} The control parameters $k$, $\kappa$, $\epsilon$, and $\gamma$ are only required to be positive and the matrix $\mathbold{S}$ should be  symmetric positive definite. The parameters $\lambda_1, \lambda_2, \ldots, \lambda_{n-1}$ must be selected such that $\mathrm{pol}(\lambda_i,k)$ is Hurwitz stable. Whilst satisfying these requirements, the control parameters can be adjusted to tune the control performance. For instance, larger $k$, $\kappa$, $\lambda_1, \lambda_2, \ldots, \lambda_{n-1}$, $\gamma$, larger entries in $\mathbold{S}$, and smaller $\epsilon$ are usually associated with shorter times for convergence of the states to the reference signal.

\begin{remark}
It should be noted that in noninvasive adaptive control strategies existing in the literature (e.g., see \cite{LiJVC:20} and \cite{LiNonDyn:22}), the term $\mathbold{F}(\mathbold{\xi}(t))\hat{\mathbold{\theta}}(t)$ is used in the control input. In those studies, if $\hat{\mathbold{\theta}}(t)$ converges to $\mathbold{\theta}$ (in the case of persistent excitation), the term $-\mathbold{F}(\mathbold{\xi}(t))\hat{\mathbold{\theta}}(t)$ in the controller compensates for the unknown nonlinear term $\mathbold{F}(\mathbold{\xi}(t))\mathbold{\theta}$.
Therefore, in those studies, the noninvasiveness of
the control input can be guaranteed by tracking a reference $\mathbold{\xi}^r(t)$.
 However, the
persistent excitation condition considered in the existing results in the literature can be difficult to satisfy.
In this paper, according to \eqref{uprime} and \eqref{eta}, the term $\mathbold{F}(\mathbold{\xi}(t))\hat{\mathbold{\theta}}(t)$  is not directly used as an additive term in the control input. Hence, it is not necessary for $\hat{\mathbold{\theta}}(t)$ to converge to the true parameters $\mathbold{\theta}$, and therefore the persistent excitation is not required. 
Such a term appears only in the gain $g(\mathbold{\xi}(t),\mathbold{\xi}^r(t),\hat{\mathbold{\theta}}(t))$ in front of $\mathrm{sat}(\mathbold{y}(t)/\epsilon)$ and not directly as an additive term in $\mathbold{u}'(t)$. As such, the term $g(\mathbold{\xi}(t),\mathbold{\xi}^r(t),\hat{\mathbold{\theta}}(t))\mathrm{sat}(\mathbold{y}(t)/\epsilon)$ converges to zero if we guarantee that $\mathbold{y}(t)$ converges to zero. Hence, whether $\hat{\mathbold{\theta}}(t)$ converges to $\mathbold{\theta}$ or not, asymptotic tracking and the noninvasiveness of the control input can be guaranteed.
\end{remark}

\begin{theorem}\label{theoremmain}
Consider the nonlinear system \eqref{system} with responses $\mathbold{\xi}^*(t)$ to the excitation  $\mathbold{\sigma}(t)$. Let $\mathbold{u}(t)=\mathbold{\sigma}(t)+\mathbold{u}'(t)$ where $\mathbold{u}'(t)$ is the control strategy defined in \eqref{uprime} with a bounded smooth reference signal $\mathbold{\xi}^r(t)$. Under this condition, if $\mathbold{\xi}^r(t)=\mathbold{\xi}^*(t)$,
             \begin{equation}\label{obj}
               \lim_{t\rightarrow \infty} (\mathbold{\xi}(t)-\mathbold{\xi}^*(t))=\mathbf{0}_{np},
             \end{equation}
while $\mathbold{u}'(t)$ is bounded and noninvasive. Moreover, if $\mathbold{\xi}^r(t)$ is not a response of the system to $\mathbold{\sigma}(t)$; then,  while $\mathbold{\xi}(t)$ remains bounded,  a stable equilibrium of the tracking error system implies nonzero tracking error and nonzero control input.
\end{theorem}

\begin{proof} The proof of the theorem is presented in two parts. In the first part, we consider the case when $\mathbold{\xi}^r(t)=\mathbold{\xi}^*(t)$, and in the second part the case when $\mathbold{\xi}^r(t)$ is not a natural response of the system to $\mathbold{\sigma}(t)$ is considered (i.e., $\mathbold{\xi}^r(t) \neq \mathbold{\xi}^*(t)$).

\textbf{Part 1:}
We prove the first part of the theorem in two steps. In the first step, we show that for $\mathbold{\xi}^r(t)=\mathbold{\xi}^*(t)$,
$\mathbold{z}(t)$ remains bounded and converges to $\mathbold{x}^{*(n-1)}(0)$, and  $\hat{\mathbold{\theta}}(t)$, $\mathbold{\xi}(t)$, $\mathbold{F}(\mathbold{\xi}(t))$, and $\dot{\mathbold{z}}(t)$ remain bounded. Then, in the next step, the achievement of the objective \eqref{obj} along with the boundedness of $\phi(t)$ is investigated. Then, we conclude that $\mathbold{u}'(t)$ remains bounded and converges to zero.\\
Step 1--By considering  the definition of $\mathbold{z}(t)$ in \eqref{z} and since $\mathbold{u}(t)=\mathbold{\sigma}(t)+\mathbold{u}'(t)$, along \eqref{system} one gets
\begin{equation}\label{zdot}
   \dot{\mathbold{z}}(t)=\mathbold{F}(\mathbold{\xi}(t))\mathbold{\theta}+\mathbold{u}'(t)-\mathbold{F}(\mathbold{\xi}(t))\hat{\mathbold{\theta}}(t)-\mathbold{\eta}(t).
\end{equation}
To analyze \eqref{zdot} under the control strategy \eqref{uprime}, we consider the following Lyapunov candidate:
\begin{equation}\label{V}
  V(t)=\frac{1}{2}\widetilde{\mathbold{z}}^\top(t) \widetilde{\mathbold{z}}(t) +\frac{1}{2}\widetilde{\mathbold{\theta}}^\top(t) \mathbold{S}^{-1} \widetilde{\mathbold{\theta}}(t),
\end{equation}
where $\widetilde{\mathbold{z}}(t)=\mathbold{z}(t)-\mathbold{x}^{r(n-1)}(0)$ and $\widetilde{\mathbold{\theta}}(t)=\hat{\mathbold{\theta}}(t)-\mathbold{\theta}$. Since $\mathbold{\theta}$ is constant, the time-derivation of $V(t)$ along  \eqref{uprime}, \eqref{thetahatdot}, and \eqref{zdot}, after some simplification yields
\begin{equation}\label{Vdot}
  \dot{V}(t)= -k\widetilde{\mathbold{z}}^\top(t)\widetilde{\mathbold{z}}(t).
\end{equation}
From \eqref{V} and \eqref{Vdot}, it follows that $\widetilde{\mathbold{z}}(t)$, $\mathbold{z}(t)$, $\widetilde{\mathbold{\theta}}(t)$, and $\hat{\mathbold{\theta}}(t)$ remain bounded.
Since  $\mathbold{\xi}^r(t)=\mathbold{\xi}^*(t)$, by considering \eqref{z} and \eqref{xstarnm1} and according to the definition of $\mathbold{e}(t)$, one gets
\begin{equation}\label{ztilde}
 \widetilde{\mathbold{z}}(t)=\mathbold{e}^{(n-1)}(t)-\int_{0}^{t}\Big(\mathbold{F}(\mathbold{\xi}(\tau))\hat{\mathbold{\theta}}(\tau)-\mathbold{F}(\mathbold{\xi}^r(\tau))\mathbold{\theta}+ \mathbold{\eta}(\tau)\Big)d\tau.
\end{equation}
By adding and subtracting $\int_{0}^{t}\mathbold{F}(\mathbold{\xi}^r(\tau))\hat{\mathbold{\theta}}(\tau)d\tau$ to the right hand side of \eqref{ztilde} and according to the definition of $\widetilde{\mathbold{\theta}}(t)$, one observes that
\begin{eqnarray}\label{ztilde2}
 \widetilde{\mathbold{z}}(t)&=&\mathbold{e}^{(n-1)}(t)-\int_{0}^{t}\Big(\mathbold{F}(\mathbold{\xi}^r(\tau))\widetilde{\mathbold{\theta}}(\tau)+\big(\mathbold{F}(\mathbold{\xi}(\tau))-\mathbold{F}(\mathbold{\xi}^r(\tau))\big)\hat{\mathbold{\theta}}(\tau) +\mathbold{\eta}(\tau)\Big)d\tau.
\end{eqnarray}
 By substituting $\mathbold{\eta}(t)$ defined in \eqref{eta} into \eqref{ztilde2}, and according to the definition of $\mathbold{y}(t)$ in \eqref{y} one gets
\begin{eqnarray}\label{ztilde3}
 \widetilde{\mathbold{z}}(t)&=&\mathbold{y}(t)-\int_{0}^{t}\Big(\mathbold{F}(\mathbold{\xi}^r(\tau))\widetilde{\mathbold{\theta}}(\tau)+\big(\mathbold{F}(\mathbold{\xi}(\tau))-\mathbold{F}(\mathbold{\xi}^r(\tau))\big)\hat{\mathbold{\theta}}(\tau) -\phi(\tau) \mathbold{y}(\tau) \nonumber\\&-&g(\mathbold{\xi}(\tau),\mathbold{\xi}^r(\tau),\hat{\mathbold{\theta}}(\tau))\mathrm{sat}(\mathbold{y}(\tau)/\epsilon)\Big)d\tau-\lambda_{n-1}\mathbold{e}^{(n-2)}(0)-\cdots-\lambda_1\mathbold{e}(0).
\end{eqnarray}
 Now, by decomposing the $\mathbold{y}(t)$ into its $p$ entries as
 \begin{equation*}
 \mathbold{y}(t)= \begin{bmatrix}
                   y_1(t) & y_2(t) & \cdots & y_p(t)
                 \end{bmatrix}^\top,
\end{equation*}
we consider two cases. The first is when $|y_k(t)|\leq \epsilon,\forall k\in\{1,2,\ldots,p\}$, and the second is when $|y_k(t)|> \epsilon$ for some $k\in\{1,2,\ldots,p\}$. These two cases are discussed below:
\begin{itemize}
  \item [] Case 1: If $|y_k(t)|\leq \epsilon,\forall k\in\{1,2,\ldots,p\},$ according to the definition of $\mathbold{y}(t)$ in \eqref{y} and since the polynomial $\mathrm{pol}(\lambda_i,n-1)$ is Hurwitz stable, $\mathbold{e}^{(k)}(t),k\in\{0,1,\ldots,n-1\},$ are bounded (due to input-to-state stability of Hurwitz linear systems \cite{Khalil}). In this condition, since $\mathbold{\xi}^r(t)=\begin{bmatrix}
                  \mathbold{x}^{r(n-1)\top }(t)  & \ldots &   \dot{\mathbold{x}}^{r\top}(t)  &   \mathbold{x}^{r\top}(t)
                 \end{bmatrix}^\top$ is bounded,  $\mathbold{\xi}(t)=\begin{bmatrix}
                  \mathbold{x}^{(n-1)\top }(t)  & \ldots &   \dot{\mathbold{x}}^{\top}(t)  &   \mathbold{x}^{\top}(t)
                 \end{bmatrix}^\top$ is bounded as well. Then, as $\mathbold{F}(\cdot)$ is a locally Lipschitz function, the boundedness of $\mathbold{\xi}(t)$ implies the boundedness of $\mathbold{F}(\mathbold{\xi}(t))$.
     By considering  \eqref{uprime} and \eqref{zdot}, we have
   \begin{equation}\label{zdot2}
   \dot{\mathbold{z}}(t)=\mathbold{F}(\mathbold{\xi}(t))\mathbold{\theta}-k\widetilde{\mathbold{z}}(t)-\mathbold{F}(\mathbold{\xi}(t))\hat{\mathbold{\theta}}(t).
\end{equation}
Hence, as $\mathbold{F}(\mathbold{\xi}(t))$, $\hat{\mathbold{\theta}}(t)$, and $\widetilde{\mathbold{z}}(t)$ are bounded, from \eqref{zdot2}
     the boundedness of $\dot{\mathbold{z}}(t)$ can be concluded.

  \item [] Case 2: If there exist $ k\in\{1,2,\ldots,p\}$ such that   $|y_k(t)|> \epsilon$, according to the definition of the saturation function,  $\mathrm{sat}(y_k(t)/\epsilon)=y_k(t)/|y_k(t)|$. Let us decompose $\big(\mathbold{F}(\mathbold{\xi}(t))-\mathbold{F}(\mathbold{\xi}^r(t))\big)\hat{\mathbold{\theta}}(t)$ to $p$ entries as follows:
    \begin{equation*}
  \big(\mathbold{F}(\mathbold{\xi}(t))-\mathbold{F}(\mathbold{\xi}^r(t))\big)\hat{\mathbold{\theta}}(t)=\begin{bmatrix}
 \mathfrak{F}_1(t) & \mathfrak{F}_2(t) & \ldots & \mathfrak{F}_p(t)
 \end{bmatrix}^\top.
    \end{equation*}
By considering \eqref{g} and according to the definition of $\mathfrak{F}_k(t)$, $\mathrm{sat}(y_k(t)/\epsilon)=y_k(t)/|y_k(t)|$ implies that there exists $\beta_k(t)\in \mathbb{R}_{>0}$, such that
      \begin{equation}\label{betakyk}
        -\mathfrak{F}_k(t)+g(\mathbold{\xi}(t),\mathbold{\xi}^r(t),\hat{\mathbold{\theta}}(t))\mathrm{sat}(y_k(t)/\epsilon)=\beta_k(t) y_k(t).
      \end{equation}
Now, let us decompose $\widetilde{\mathbold{z}}(t)$, $\mathbold{F}(\mathbold{\xi}^r(t))\widetilde{\mathbold{\theta}}(t)$, and $\mathbold{e}(t)$ as follows:
 \begin{equation*}
 \begin{split}
\widetilde{ \mathbold{z}}(t)&= \begin{bmatrix}
                   \widetilde{z}_1(t) & \widetilde{z}_2(t) & \cdots & \widetilde{z}_p(t)
                 \end{bmatrix}^\top\\
   \mathbold{F}(\mathbold{\xi}^r(t))\widetilde{\mathbold{\theta}}(t)&=\begin{bmatrix}
  \mathfrak{F}'_1(t) & \mathfrak{F}'_2(t) & \ldots & \mathfrak{F}'_p(t)
  \end{bmatrix}^\top\\
 \mathbold{e}(t)&= \begin{bmatrix}
                   e_1(t) & e_2(t) & \cdots & e_p(t)
                 \end{bmatrix}^\top.
\end{split}
\end{equation*}
By considering \eqref{betakyk}, from \eqref{ztilde3} one gets
\begin{eqnarray}
 y_k(t)&=&\int_{0}^{t}\Big(\mathfrak{F}'_k(\tau)-\beta_k(\tau) y_k(\tau)-\phi(\tau) y_k(\tau)\Big)d\tau \nonumber \\&+&\lambda_{n-1}e_k^{(n-2)}(0)+\cdots+\lambda_1e_k(0)\label{zktilde}+\widetilde{z}_k(t).
\end{eqnarray}
Since  $\mathfrak{F}'_k(t)$ and $ \widetilde{z}_k(t)$ are bounded and $\beta_k(t)+\phi(t)$ is positive, by considering \eqref{zktilde}, it follows that $y_k(t)$ remains bounded. Hence, based on the arguments given for Case 1, the boundedness of $\mathbold{\xi}(t)$ and $\mathbold{F}(\mathbold{\xi}(t))$ and then the boundedness of  $\dot{\mathbold{z}}(t)$ can be concluded.
\end{itemize}
According to the aforementioned two cases, $\mathbold{\xi}(t)$, $\mathbold{F}(\mathbold{\xi}(t))$, and $\dot{\mathbold{z}}(t)$ are bounded.
By considering \eqref{Vdot},
$\dot{V}(t)$ is negative semidefinite. Since $V(t)$ is lower bounded and $\dot{V}(t)$ is nonpositive, the limit of $V(t)$ exists, and $V(t)$ converges to  a constant value. The convergence of $V(t)$ to a constant value means that in the time period $[t,t+s],t\to \infty,$ for any real positive constant  $s$, we should have

\begin{equation}\label{Vdotint}
	\lim_{t\rightarrow \infty}\int_{t}^{t+s} \dot{V}(\tau)d\tau=0.
\end{equation}
Equation \eqref{Vdotint} means that $\dot{V}(t)$ converges to zero in the time period $[t,t+s],t\to \infty,$ unless $\ddot{V}(t)$ is unbounded for bounded $\dot{V}(t)$, but such discontinuity is not possible according to \eqref{zdot2}\footnote{This can be also concluded by the Barbalat lemma \cite{TaoTAC:97} and the invariance principle for nonautonomous systems \cite{BarkanaIJC:14}.}.
Therefore, $\dot{V}(t)$ converges to zero which implies that $\widetilde{\mathbold{z}}(t)$ converges to zero.\\ 
Step 2--From \eqref{ztilde3} one gets
\begin{eqnarray}\label{ztilde4}
 \dot{\mathbold{y}}(t)&=&\big(\mathbold{F}(\mathbold{\xi}(t))-\mathbold{F}(\mathbold{\xi}^r(t))\big)\hat{\mathbold{\theta}} -\phi(t) \mathbold{y}(t) -g(\mathbold{\xi}(t),\mathbold{\xi}^r(t),\hat{\mathbold{\theta}}(t))\mathrm{sat}(\mathbold{y}(t)/\epsilon)\nonumber \\&+&\dot{\mathbold{z}}(t)+\mathbold{F}(\mathbold{\xi}^r(t))\widetilde{\mathbold{\theta}}(t).
\end{eqnarray}
Since $\widetilde{\mathbold{z}}(t)$  converges to zero,
for any $s \in \mathbb{R}_{>0}$,
\begin{equation}\label{limdz}
	\lim_{t\rightarrow \infty}\int_{t}^{t+s} \dot{\mathbold{z}}(\tau)d\tau=0.
\end{equation}
Equation \eqref{limdz} means that $\dot{\mathbold{z}}(t)$ converges to zero in the time period $[t,t+s],t\to \infty,$ unless $\ddot{\mathbold{z}}(t)$ is unbounded for bounded $\dot{\mathbold{z}}(t)$, but such discontinuity is not possible according to  the nonlinear system \eqref{zdot2}.
Hence, the convergence of $\dot{\mathbold{z}}(t)$ to zero can be concluded.
Since $\dot{\mathbold{z}}(t)$ and $\widetilde{\mathbold{z}}(t)$ remain bounded and converge to zero, from \eqref{zdot2} it follows that $\mathbold{F}(\mathbold{\xi}(t))\widetilde{\mathbold{\theta}}(t)$ remains bounded and converges to zero. By adding and subtracting $\mathbold{F}(\mathbold{\xi}(t))\widetilde{\mathbold{\theta}}(t)$ to the right side of \eqref{ztilde4} and since $\widetilde{\mathbold{\theta}}(t)=\hat{\mathbold{\theta}}(t)-\mathbold{\theta}$, after some manipulation one gets
\begin{eqnarray}\label{ztilde5}
 \dot{\mathbold{y}}(t)=\big(\mathbold{F}(\mathbold{\xi}(t))-\mathbold{F}(\mathbold{\xi}^r(t))\big)\mathbold{\theta} -\phi(t) \mathbold{y}(t) -g(\mathbold{\xi}(t),\mathbold{\xi}^r(t),\hat{\mathbold{\theta}}(t))\mathrm{sat}(\mathbold{y}(t)/\epsilon)+\mathbold{w}(t),
\end{eqnarray}
where $\mathbold{w}(t)=\dot{\mathbold{z}}(t)+\mathbold{F}(\mathbold{\xi}(t))\widetilde{\mathbold{\theta}}(t)$ is bounded and converges to zero.
As $\mathbold{\xi}(t)$ and $\mathbold{\xi}^r(t)$ are bounded and $\mathbold{F}(\cdot)$ is locally Lipschitz, there exists a finite constant $\ell_1\in \mathbb{R}_{>0}$ such that
\begin{equation}\label{ineq1}
 \big \|\mathbold{F}(\mathbold{\xi}(t))-\mathbold{F}(\mathbold{\xi}^r(t))\big\|\leq \ell_1 \|\mathbold{\xi}(t)-\mathbold{\xi}^r(t)\|.
\end{equation}
Moreover, according to \eqref{y}, since the polynomial $\mathrm{pol}(\lambda_i,n-1)$ is Hurwitz stable, there exists a vanishing bounded function $h\left(t,\mathbold{e}(0),\ldots,\mathbold{e}^{(n-2)}(0)\right)$ and a finite constant
$\ell_2\in \mathbb{R}_{>0}$
such that \cite{Khalil}
\begin{equation}\label{ineq2}
 \|\mathbold{\xi}(t)-\mathbold{\xi}^r(t)\|\leq h+\ell_2\|\mathbold{y}(t)\|.
\end{equation}
Hence, by considering \eqref{ineq1} and \eqref{ineq2}, one gets
\begin{equation}\label{ineq3}
	\big \|\mathbold{F}(\mathbold{\xi}(t))-\mathbold{F}(\mathbold{\xi}^r(t))\big\|\leq \ell_1h+\ell_1\ell_2\|\mathbold{y}(t)\|.
\end{equation}
Now, we consider the following Lyapunov candidate:
\begin{equation*}
  V'(t)=\frac{1}{2}\mathbold{y}^\top(t)\mathbold{y}(t) +\frac{1}{2}\gamma^{-1}(\phi(t)-\phi^*)^2,
\end{equation*}
where
$\phi^* \in \mathbb{R}_{>0}$ is a constant that satisfies the following inequality for an arbitrary $\varrho \in \mathbb{R}_{>0}$,
\begin{equation}\label{phistar}
  \phi^* > \frac{\varrho}{2}+\ell_1\ell_2\|\mathbold{\theta}\|.
\end{equation}
The time derivation of $V'(t)$ along \eqref{phidot} and \eqref{ztilde5} yields
\begin{eqnarray}\label{V2dot}
  \dot{V}'(t)&=&\mathbold{y}^\top(t)\big(\mathbold{F}(\mathbold{\xi}(t))-\mathbold{F}(\mathbold{\xi}^r(t))\big)\mathbold{\theta} -\phi(t) \mathbold{y}^\top(t)\mathbold{y}(t) - g(\mathbold{\xi}(t),\mathbold{\xi}^r(t),\hat{\mathbold{\theta}}(t))\mathbold{y}^\top(t)\mathrm{sat}(\mathbold{y}(t)/\epsilon)\nonumber\\&+& \mathbold{y}^\top(t)\mathbold{w}(t)+(\phi(t)-\phi^*)\mathbold{y}^\top(t)\mathbold{y}(t).
\end{eqnarray}
According to the Young inequality, one observes that
\begin{equation}\label{Young}
	\begin{split}
		\mathbold{y}^\top(t)\mathbold{w}(t)&\leq \frac{\varrho}{2}\mathbold{y}^\top(t)\mathbold{y}(t) + \frac{1}{2\varrho}\mathbold{w}^\top(t)\mathbold{w}(t).
	\end{split}
\end{equation}
By considering \eqref{ineq3} and \eqref{Young},   \eqref{V2dot} can be simplified as follows:
\begin{equation}\label{V2dot3}
	\dot{V}'(t)\leq \ell_1h\|\mathbold{\theta}\|\|\mathbold{y}(t)\|+\ell_1\ell_2 \|\mathbold{\theta}\|\mathbold{y}^\top(t)\mathbold{y}(t)+\frac{\varrho}{2}\mathbold{y}^\top(t)\mathbold{y}(t) + \frac{1}{2\varrho}\mathbold{w}^\top(t)\mathbold{w}(t)
	-\phi^* \mathbold{y}^\top(t)\mathbold{y}(t),
\end{equation}
where $\ell_1h\|\mathbold{\theta}\|\|\mathbold{y}(t)\|$ is bounded and converges to zero. Note that as  $\mathbold{\xi}(t)$ and $\mathbold{\xi}^r(t)$ are bounded, $\mathbold{y}(t)$ is also bounded.
From  \eqref{phistar} and \eqref{V2dot3}, it follows that there exists $\chi \in \mathbb{R}_{>0}$ such that
\begin{equation}\label{V2dot4}
  \dot{V}'(t)\leq - \chi\mathbold{y}^\top(t) \mathbold{y}(t)  +\ell_1h\|\mathbold{\theta}\|\|\mathbold{y}(t)\|+\frac{1}{2\varrho}\mathbold{w}^\top(t)\mathbold{w}(t).
\end{equation}
Since $\ell_1h\|\mathbold{\theta}\|\|\mathbold{y}(t)\|+1/(2\varrho)\mathbold{w}^\top(t)\mathbold{w}(t)$ is bounded and converges to zero, according to \eqref{V2dot4},  for any nonzero $\mathbold{y}(t)$ there exists a finite time $t_f$ such that for $t\geq t_f$,  $\dot{V}'(t)$ is negative semidefinite. Note that according to  \eqref{phidot} and \eqref{ztilde5}, $V'(t)$ cannot become unbounded in finite time. Therefore, as  $\dot{V}'(t)$ is negative semidefinite for $t\geq t_f$, $V'(t)$ is bounded at all times.
Hence, the limit of $V'(t)$ exists, that is, for any $s \in \mathbb{R}_{>0}$,
\begin{equation}\label{V2dotint}
	\lim_{t\rightarrow \infty}\int_{t}^{t+s} \dot{V}'(\tau)d\tau=0.
\end{equation}
Equation \eqref{V2dotint} means that $\dot{V}'(t)$ converges to zero in the time period $[t,t+s],t\to \infty,$ unless $\ddot{V}'(t)$ is unbounded for bounded $\dot{V}'(t)$ which is not possible according to \eqref{phidot} and \eqref{ztilde5}.
Hence, $\dot{V}'(t)$ converges to zero which implies that $\mathbold{y}(t)$ converges to zero.
According to \eqref{y}, by convergence of $\mathbold{y}(t)$  to zero, since the polynomial $\mathrm{pol}(\lambda_i,n-1)$ is Hurwitz stable, $\mathbold{e}^{(k)}(t),k\in\{0,1,\ldots,n-1\},$ converge to zero. Therefore,
the objective \eqref{obj} is satisfied. Moreover, since $V'(t)$ is bounded, $\phi(t)$ remains bounded. Hence, as $\mathbold{z}(t)$, $\mathbold{y}(t)$, and $\mathbold{\xi}(t)$ are bounded, $\mathbold{u}'(t)$ is bounded. Now, we need to show that $\mathbold{u}'(t)$ converges to zero. We have shown that $\mathbold{e}^{(k)}(t),k\in\{0,1,\ldots,n-1\},$ and $\mathbold{y}(t)$ converge to zero.
                It was also shown that  $\widetilde{\mathbold{z}}(t)$ converges to zero.
                Therefore, according to \eqref{uprime} and \eqref{eta}, $\mathbold{u}'(t)$ converges to zero.
                
\textbf{Part 2:}
 If $\mathbold{\xi}^r(t)$ is not a  response of the system to $\mathbold{\sigma}(t)$, from \eqref{systemopenloop} it follows that
\begin{equation}\label{systemopenloop2}
  \mathbold{x}^{r(n)}(t)=\mathbold{F}(\mathbold{\xi}^r(t))\mathbold{\theta}+\mathbold{\sigma}(t)+\mathbold{\Delta}(t),
\end{equation}
where $\mathbold{\Delta}(t)\in \mathbb{R}^{p}$ is bounded and not identical to zero.
According to \eqref{systemopenloop2},  one gets
\begin{equation}\label{xstarnm12}
 \mathbold{x}^{r(n-1)}(0)=\mathbold{x}^{r(n-1)}(t)-\int_{0}^{t}\Big(\mathbold{F}(\mathbold{\xi}^r(\tau))\mathbold{\theta}+\mathbold{\Delta}(\tau) +\mathbold{\sigma}(\tau)\Big)d\tau.
\end{equation}
In Part 1, it was show that if $|y_k(t)|\leq \epsilon,\forall k\in\{1,2,\ldots,p\}$, $\mathbold{\xi}(t)$ always remains bounded. Let us now consider the case when $|y_k(t)|> \epsilon$ for some $k\in\{1,2,\ldots,p\}$.
By decomposing $\mathbold{\Delta}(t)$ as
 \begin{equation*}
		\mathbold{\Delta}(t)= \begin{bmatrix}
			\Delta_1(t) & \Delta_2(t) & \cdots & \Delta_p(t)
		\end{bmatrix}^\top,
\end{equation*}
and by considering \eqref{xstarnm12}, \eqref{zktilde} yields
\begin{eqnarray}
	\widetilde{z}_k(t)&=&y_k(t)-\int_{0}^{t}\Big(\mathfrak{F}'_k(\tau)-\beta_k(\tau) y_k(\tau)-\phi(\tau) y_k(\tau)-\Delta_k(\tau)\Big)d\tau \nonumber \\&-&\lambda_{n-1}e_k^{(n-2)}(0)-\cdots-\lambda_1e_k(0)\label{zktilde2}.
\end{eqnarray}
Note that whether $\mathbold{\Delta}(t)$ is identical to zero or not,  \eqref{Vdot} still is satisfied, implying that $\widetilde{\mathbold{z}}(t)$ and $\widetilde{\mathbold{\theta}}(t)$ are bounded, and hence $\widetilde{z}_k(t)$ and $\mathfrak{F}'_k(t)$ are bounded.
Since $ \widetilde{z}_k(t)$, $\mathfrak{F}'_k(t)$, and $\Delta_k(t)$  are bounded and $\beta_k(t)+\phi(t)$ is positive, by considering \eqref{zktilde2},  $y_k(t)$ remains bounded. Hence, the boundedness of $\mathbold{\xi}(t)$ can be concluded.
By considering \eqref{xstarnm12},  \eqref{ztilde5} also becomes
\begin{eqnarray}\label{ztilde52}
	\dot{\mathbold{y}}(t)&=&\big(\mathbold{F}(\mathbold{\xi}(t))-\mathbold{F}(\mathbold{\xi}^r(t))\big)\mathbold{\theta} -\phi(t) \mathbold{y}(t) -g(\mathbold{\xi}(t),\mathbold{\xi}^r(t),\hat{\mathbold{\theta}}(t))\mathrm{sat}(\mathbold{y}(t)/\epsilon)\nonumber\\&-&\mathbold{\Delta}(t)+\mathbold{w}(t).
\end{eqnarray}
Whether $\mathbold{\Delta}(t)$ is identical to zero or not, in a way similar to Part 1, one observes that $\widetilde{\mathbold{z}}(t)$, $\dot{\mathbold{z}}(t)$ and then $\mathbold{w}(t)$ converge to zero. Therefore, if $\mathbold{\Delta}(t)$ is not identical to zero, $\mathbold{y}(t)=\mathbf{0}_p$ is not the equilibrium point of the system \eqref{ztilde52}. Hence, $\mathbold{e}(t)=\mathbf{0}_p$ is not the equilibrium point of the system \eqref{ztilde52}. As a result, since the polynomial $\mathrm{pol}(\lambda_i,n-1)$ is Hurwitz stable, $\mathbold{\eta}(t)=\mathbf{0}_p$ is not the stable equilibrium point of the system. Therefore,  $\mathbold{u}'(t)=\mathbf{0}_p$ does not imply a stable equilibrium for the tracking error system, that is, a stable equilibrium of the tracking error system implies nonzero control input.            
\end{proof}

According to the results of Theorem~\ref{theoremmain}, if the reference signal $\mathbold{\xi}^r(t)$ is a natural response of the system \eqref{system} to the excitation  $\mathbold{\sigma}(t)$, $\mathbold{\xi}(t)$ converges to $\mathbold{\xi}^r(t)$, while $\mathbold{u}'(t)$ is noninvasive. In the context of CBC, the behavior of the system is a priori unknown. Hence, to discover such uncontrolled behavior, the reference signal $\mathbold{\xi}^r(t)$ will be found iteratively by checking whether it makes the control input noninvasive or not. Hence, as addressed in Theorem~\ref{theoremmain}, it is desirable to guarantee that the reference signal results in noninvasive control ``only if'' it is a natural response of the nonlinear system~\eqref{system} to $\mathbold{\sigma}(t)$. In other words, in cases when the reference $\mathbold{\xi}^r(t)$ is not a natural response of the system \eqref{system} to $\mathbold{\sigma}(t)$, $\mathbold{\xi}(t)$ should not converge to $\mathbold{\xi}^r(t)$ and $\mathbold{u}'(t)$ should be invasive. The obtained results are validated by numerical simulation in the next section.

\begin{remark}
According to the proof of Theorem \ref{theoremmain}, while the proposed control strategy relies on the adaptive vector $\hat{\mathbold{\theta}}(t)$, the convergence of $\hat{\mathbold{\theta}}(t)$ to  $\mathbold{\theta}$ is not required. Accordingly, the persistent excitation of $\mathbold{\xi}^*(t)$, which is typically required for $\hat{\mathbold{\theta}}(t)$ to converge to $\mathbold{\theta}$, is not a condition for the accurate performance of the proposed control strategy. However, if $\mathbold{\xi}^*(t)$ is persistently exciting, it is straightforward to show that under the proposed control strategy, while tracking $\mathbold{\xi}^*(t)$, $\hat{\mathbold{\theta}}(t)$ asymptotically converges to $\mathbold{\theta}$.
\end{remark}

\begin{remark}
The control strategy proposed in \eqref{uprime}-\eqref{thetahatdot} is leading the system behavior toward the surface $\mathbold{y}(t)=\mathbf{0}_p$, which according to \eqref{y} implies that the system is controlled by a linear controller. Taking a second-order dynamical system as an example, the behavior of the system on $\mathbold{y}(t)=\mathbf{0}_p$ is thus similar to a system controlled by a proportional (P) or proportional-derivative (PD) controller. This type of linear controllers has been extensively used for CBC as it guarantees the existence of a noninvasive solution for $\mathbold{\xi}^r(t)$. Therefore, it is thought that the iterative methods currently used to find noninvasive reference signals in CBC will also work for the adaptive control strategy proposed in this paper. However, contrary to using P and PD controllers directly on nonlinear systems (as in the literature), the proposed adaptive control strategy has a ``guaranteed stabilizing performance''. Indeed, the surface $\mathbold{y}(t)=\mathbf{0}_p$ does not depend on unknown parameters, and the effect of all the unknown parameters is compensated for when $\lim_{t\rightarrow \infty}\mathbold{y}(t)=\mathbf{0}_p$. As along as $\mathbold{y}(t)=\mathbf{0}_p$, the tracking error $\mathbold{e}(t)$ converges to zero. 
\end{remark}

\section{Simulation Results} \label{SecSimulations}
In this section, the accuracy of the proposed control strategy is numerically demonstrated on three systems. First, a bi-stable Duffing oscillator and then two beam structures exhibiting 1:1 and 3:1 modal interactions, respectively, are considered.

\subsection{Duffing oscillator}
Consider a forced Duffing oscillator with the following dynamics:
\begin{equation*}\label{Ex1Sys}
\ddot{x}(t)=\theta_1\dot{x}(t)+\theta_2x(t)+\theta_3 x(t)^3+\sigma(t)+u'(t),
\end{equation*}
where $x(t)$ denotes the displacement of the oscillator, $\sigma(t)$ is the applied excitation, $u'(t)$ is the control input, and $\theta_1=-0.1$, $\theta_2=4$, and $\theta_3=-2$ are unknown damping, linear stiffness, and nonlinear stiffness parameters, respectively. For illustration, the response of the oscillator under harmonic excitation $\sigma(t)=0.15\cos(\omega t)$ is computed using numerical continuation and shown in Fig.~\ref{Ex1-OpenLoop}-(a). 

\begin{figure}[t]
\centering
\begin{tabular}{cc}
\subfloat[]{\includegraphics[width=0.45\textwidth]{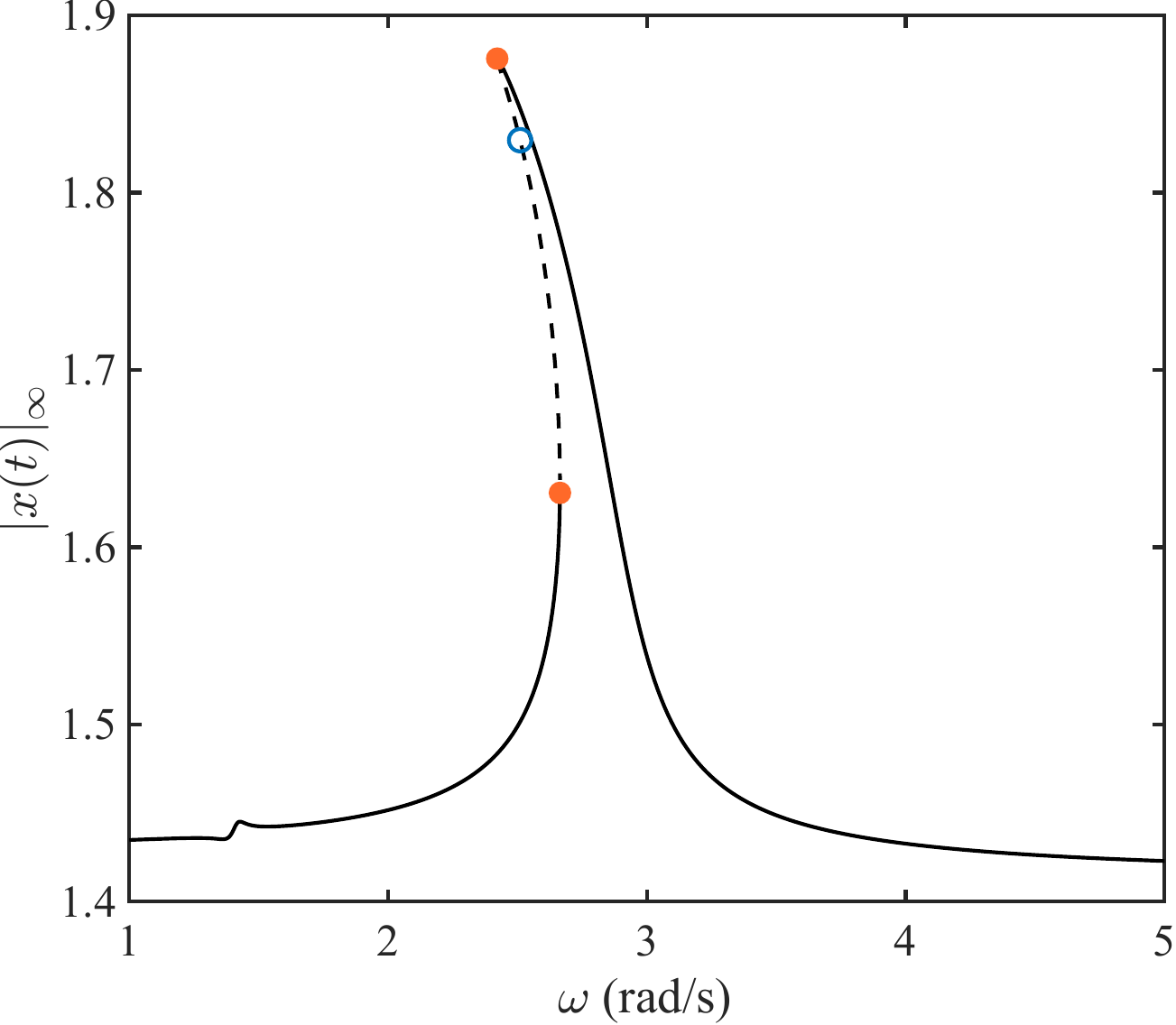}} & 
\subfloat[]{\includegraphics[width=0.45\textwidth]{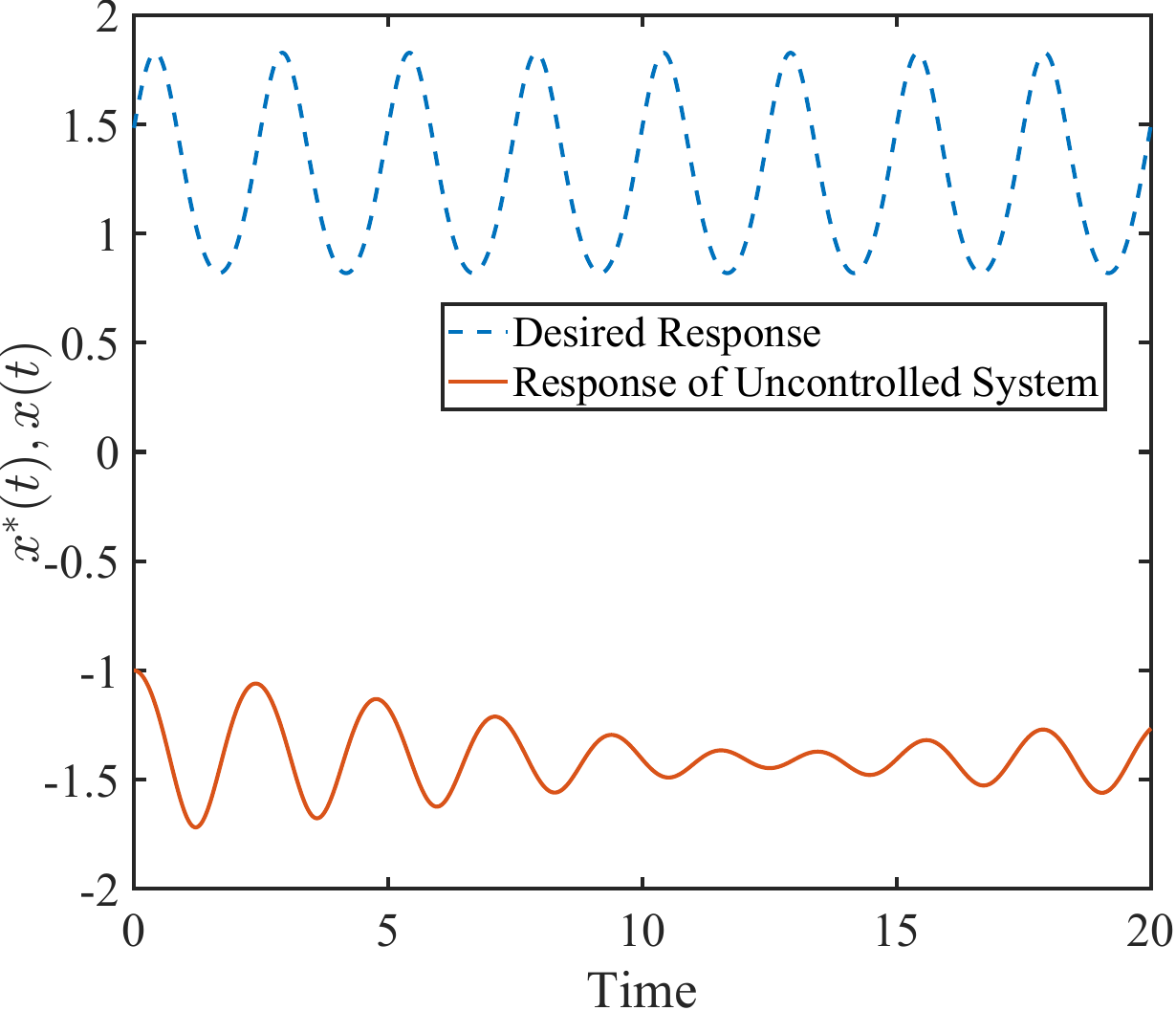}}
\end{tabular}
\caption{(a) Frequency response of the oscillator around the positive equilibrium in which the specific point associated with $\mathbold{\xi}^*(t)$ is shown by a circle ($\textcolor{matlab_blue}{\circ}$). The dashed line represents unstable responses, and dots ($\textcolor{matlab_orange} {\bullet}$) show limit-point bifurcations. (b) Desired responses of the uncontrolled oscillator for the initial states $(\dot{x}(0),x(0))=(1.314,1.483)$  and the uncontrolled response for the initial states $(\dot{x}(0),x(0))=(0,-1)$.}
\label{Ex1-OpenLoop}
\end{figure}

To illustrate the performance of the proposed control method, a reference signal corresponding to an unstable periodic response of the oscillator at $\omega=2.515$ is chosen (see the circle marker in Fig.~\ref{Ex1-OpenLoop}-(a)). This reference trajectory is given by $\mathbold{\xi}^r(t)=\mathbold{\xi}^*(t)=\begin{bmatrix}
	\dot{x}^*(t)  &   x^{*}(t)
\end{bmatrix}^\top$ where
\begin{equation*}
  x^*(t)\approx1.271+0.244\cos(\omega t)-0.026\cos(2\omega t)-0.005\cos(3\omega t)+0.436\sin(\omega t)+ 0.045\sin(2\omega t),
\end{equation*}
and $\dot{x}^*(t)$ is obtained by differentiating $x^*(t)$, and hence $(\dot{x}(0),x(0))=(1.314,1.483)$. The response of the oscillator is very sensitive to the initial states as the oscillator has two equilibria with stable steady-state oscillations around each. The desired response of the oscillator to $\sigma(t)=0.15\cos(\omega t)$ along with the response of the oscillator to $\sigma(t)=0.15\cos(\omega t)$ for the different initial states $(\dot{x}(0),x(0))=(0,-1)$ is shown in Fig.~\ref{Ex1-OpenLoop}-(b). Therefore, without control of the oscillator response, it is apparent that the oscillator response can be significantly different from the desired one. Note that $x^*(t)$ is an unstable trajectory of the uncontrolled oscillator. Therefore, it may not be precisely revealed even if the initial states of the oscillator are close to the initial states of the desired trajectory.
\begin{figure}[t]
\centering
\begin{tabular}{cc}
\subfloat[]{\includegraphics[width=0.45\textwidth]{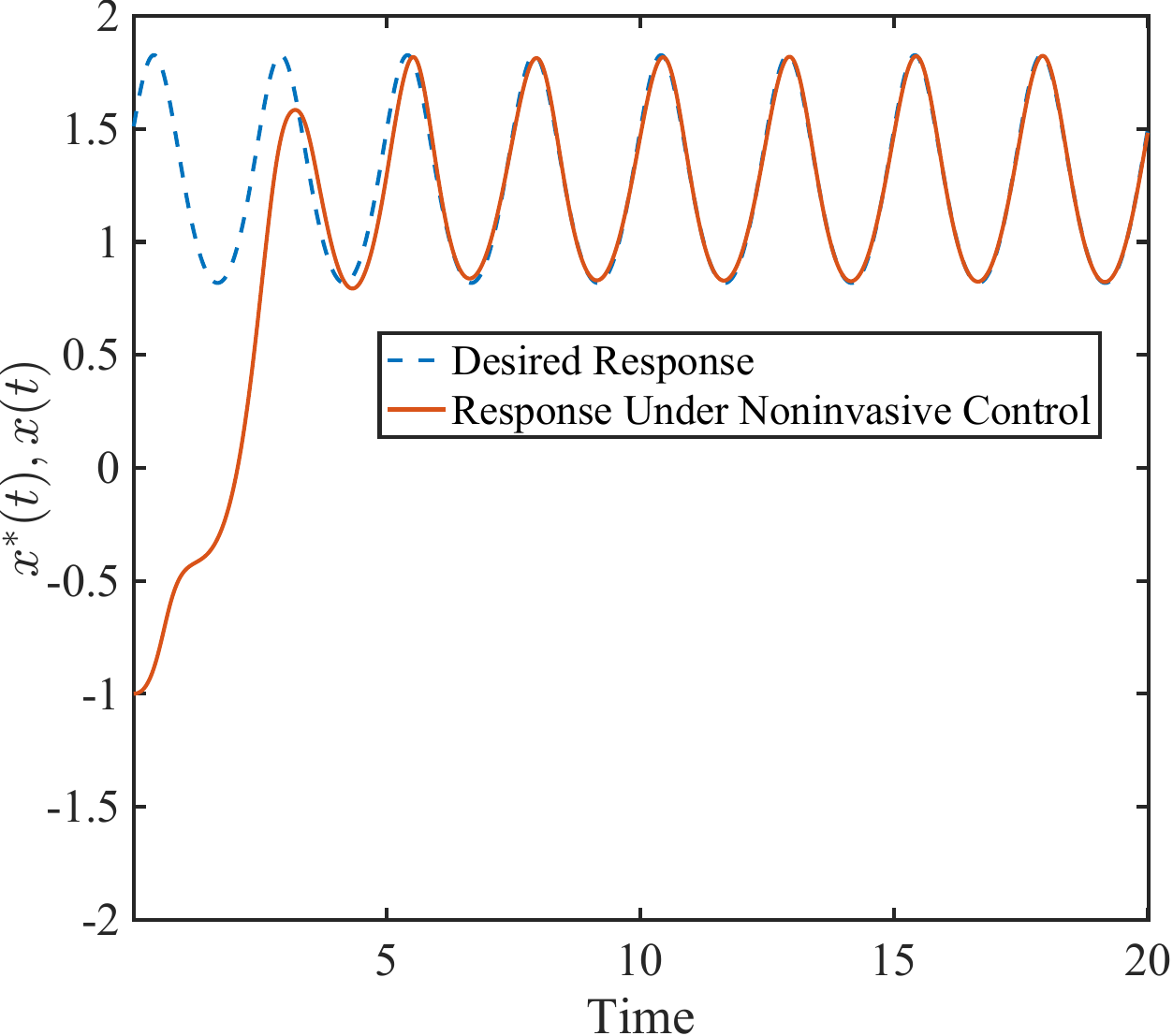}} & 
\subfloat[]{\includegraphics[width=0.44\textwidth]{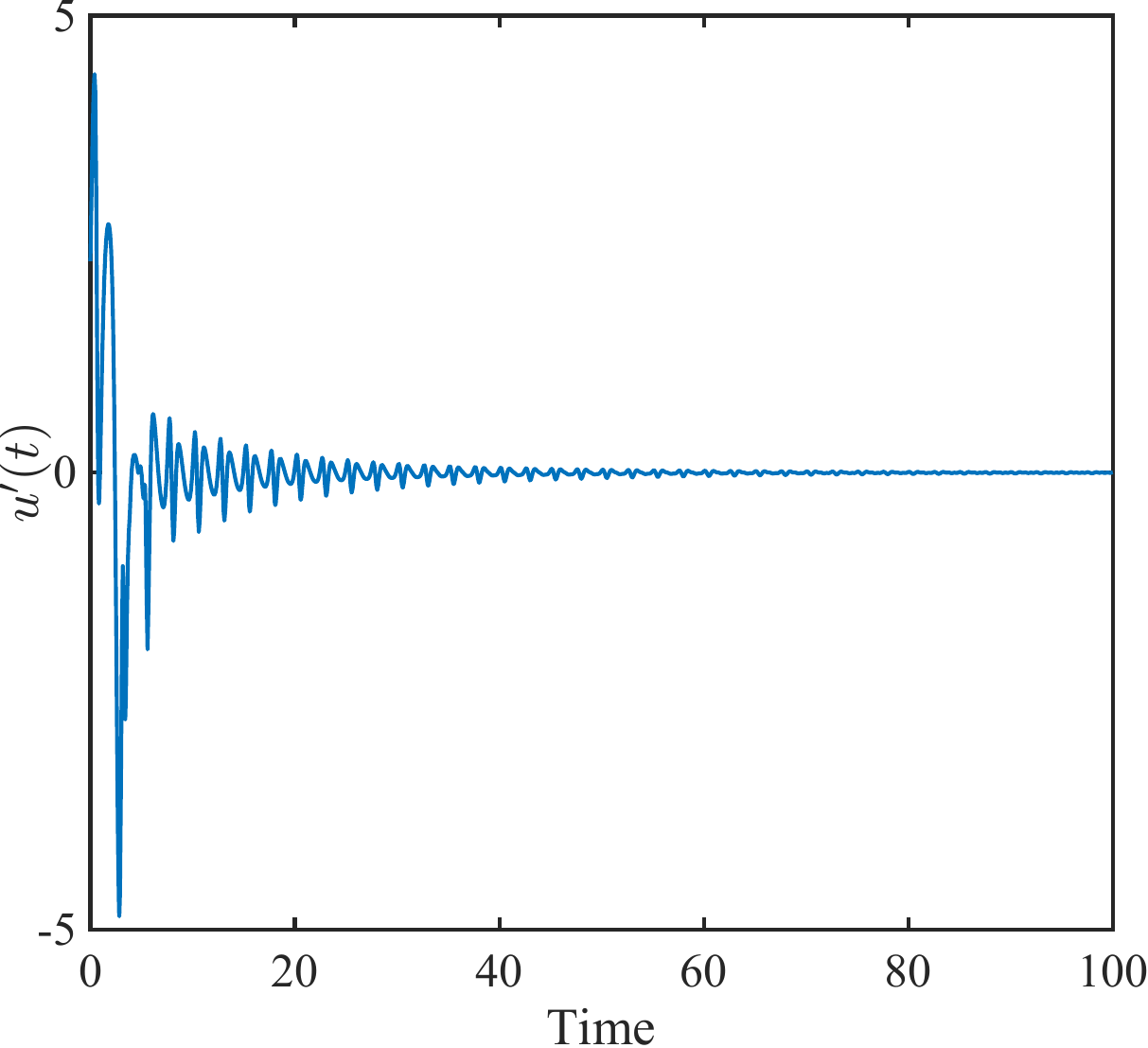}}\\
\subfloat[]{\includegraphics[width=0.45\textwidth]{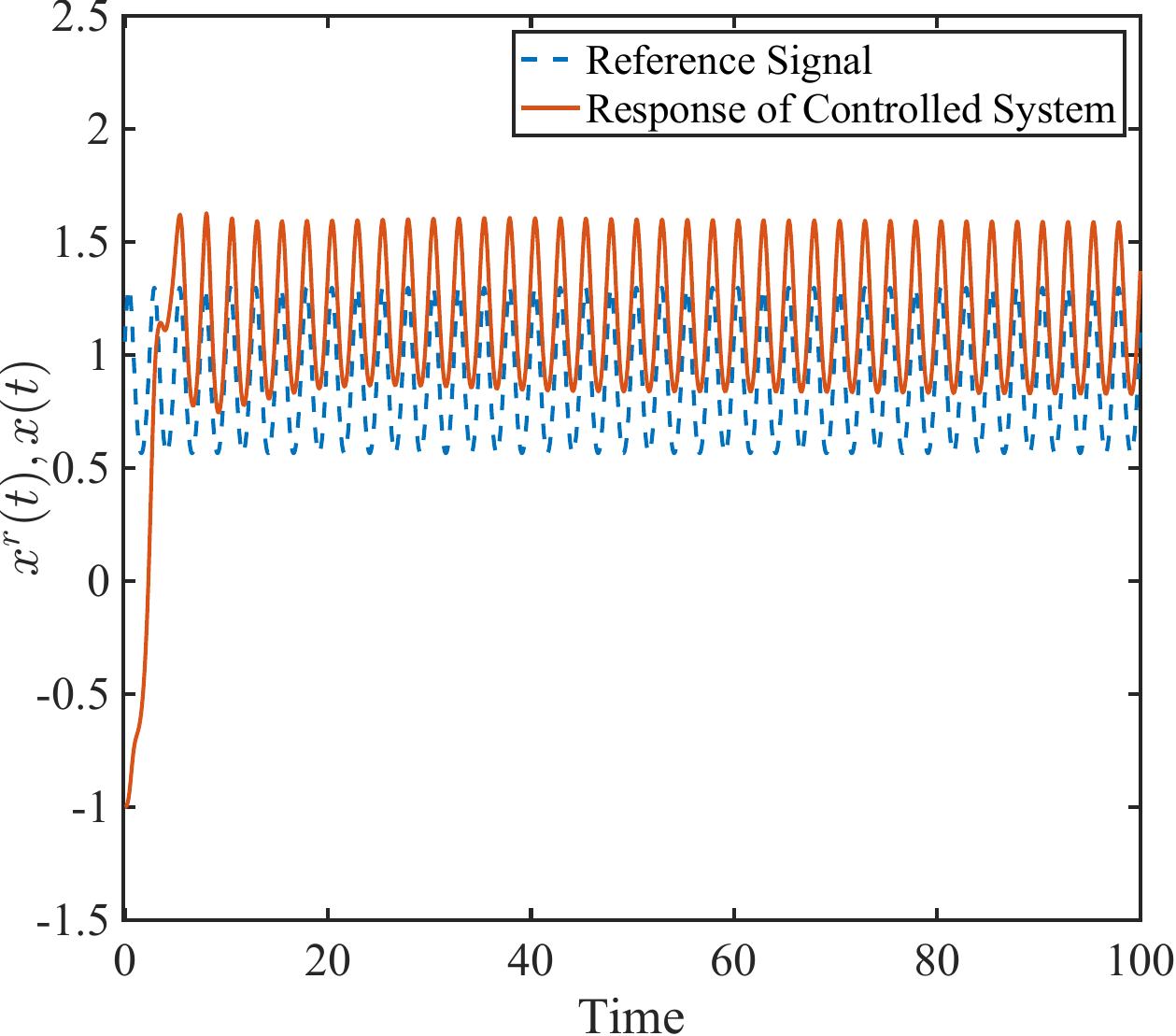}} & 
\subfloat[]{\includegraphics[width=0.445\textwidth]{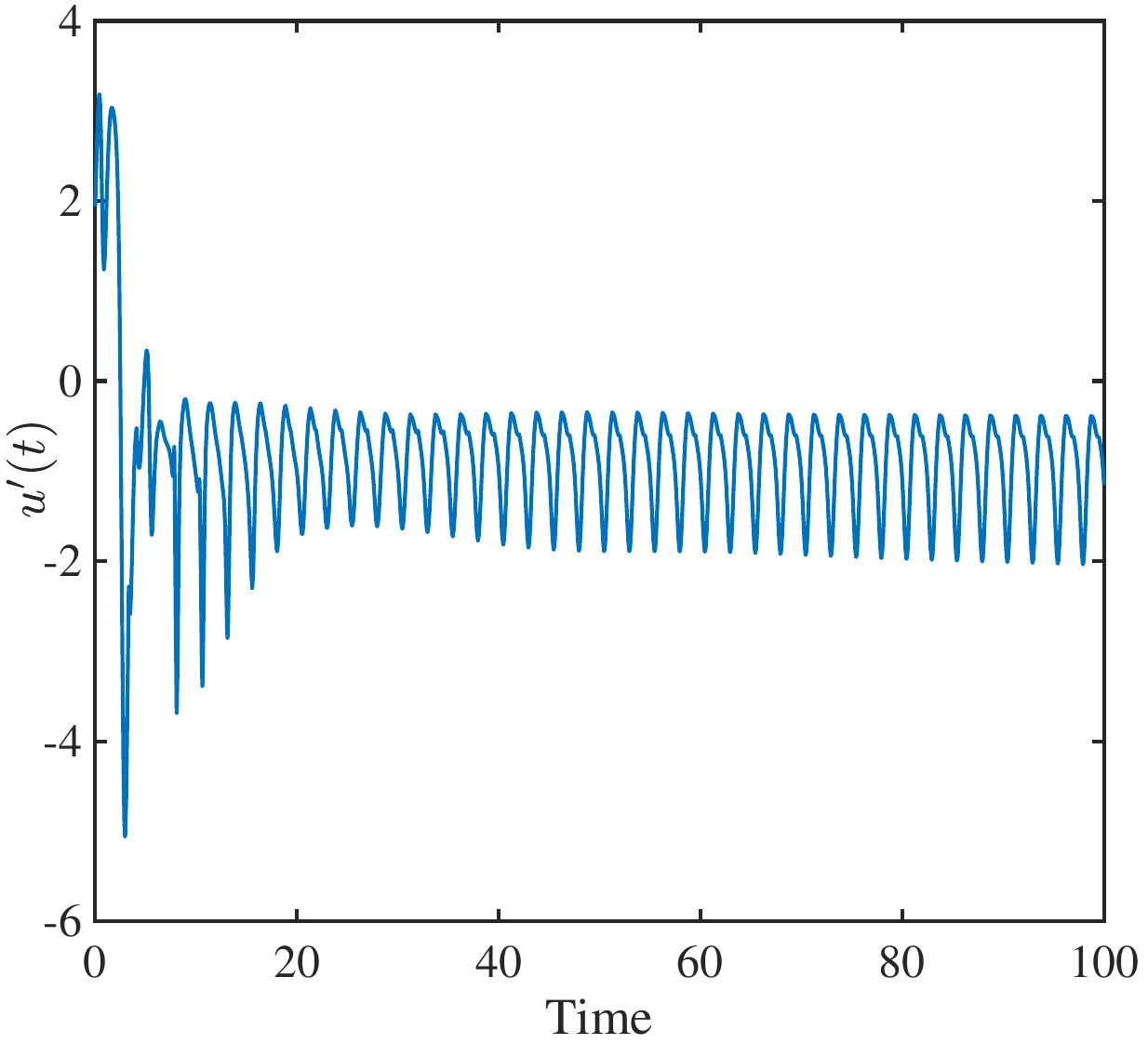}}
\end{tabular}
\caption{(a) Desired response of the oscillator with the initial states $(\dot{x}(0),x(0))=(1.314,1.483)$ and the
		controlled response of the oscillator for the initial states $(\dot{x}(0),x(0))=(0,-1)$. (b) Noninvasive control input of the oscillator.
  (c) Desired reference $x^r(t)$ (which is not a natural response of the system to $\sigma(t)$) and controlled response of the oscillator. (d) Invasive control input of the oscillator when $\mathbold{\xi}^r(t)$ is not a natural response of the system to $\sigma(t)$.}
\label{Ex1-ClosedLoop}
\end{figure}
By employing the proposed adaptive control strategy, it is guaranteed that the system response asymptotically converges to the desired response (see Fig.~\ref{Ex1-ClosedLoop}-(a)) while the control input asymptotically converges to zero, i.e., it becomes noninvasive (see Fig.~\ref{Ex1-ClosedLoop}-(b)).  
The control parameters we used for simulation are $\mathbold{S}=2\mathbold{I}_3$, $k=1$, $\kappa=1$, $\lambda_1=1$, $\gamma=0.1$, and $\epsilon=1$.

Under the proposed control strategy, $\mathbold{\xi}(t)$ does not converge to $\mathbold{\xi}^r(t)$ and $u'(t)$ is invasive if $\mathbold{\xi}^r(t)$ is not a natural response of the uncontrolled oscillator to $\sigma(t)$. To illustrate this, a simulation is performed by choosing a $\mathbold{\xi}^r(t)$ whose Fourier coefficients have up to 30 percent deviations from the coefficients of $\mathbold{\xi}^*(t)$.
The system response and control input are shown in Fig.~\ref{Ex1-ClosedLoop}-(c) and Fig.~\ref{Ex1-ClosedLoop}-(d), respectively. As expected, $\mathbold{\xi}(t)$ does not converge to $\mathbold{\xi}^r(t)$ and $u'(t)$ is invasive.

\subsection{Cross-beam structure}\label{Sec:CBS}
The cross-beam structure studied in~\cite{RensonProcA:18} and inspired by the physical system tested in~\cite{EhrhardtJSV:19} is now considered. This structure has the particularity of having its first two vibration modes very close in frequency ($\approx$ 0.5 Hz apart), leading to 1:1 mode interactions. As such, the model of the structure is a two degrees-of-freedom modal model given by
\begin{equation*}
\ddot{\mathbold{x}}(t)+\mathbold{\Xi}\dot{\mathbold{x}}(t)+\mathbold{\Lambda} \mathbold{x}(t)+\mathbold{N}(\mathbold{x}(t))=\mathbold{\sigma}(t)+\mathbold{u}'(t),
\end{equation*}
with
\begin{equation*}
\begin{split}
\mathbold{x}(t)&=\begin{bmatrix}x_1(t)\\ x_2(t)\end{bmatrix},\mathbold{\Xi}=\begin{bmatrix}
                          2\zeta_1\omega_1 & 0 \\
   0 & 2\zeta_2\omega_2
         \end{bmatrix},\mathbold{\Lambda}=\begin{bmatrix}
 \omega_1^2 & 0 \\
0 & \omega_2^2
    \end{bmatrix},\mathbold{N}(\mathbold{x}(t))=\begin{bmatrix}\mathbold{N}_1(\mathbold{x}(t))\\\mathbold{N}_2(\mathbold{x}(t))
 \end{bmatrix}\\
\mathbold{N}_1(\mathbold{x}(t))&=1/2\gamma_{11}x_1(t)^2+1/2(\gamma_{21}+\gamma_{31})x_1(t)x_2(t)+1/2\gamma_{41}x_2(t)^2+1/3\gamma_{51}x_1(t)^3\\&\quad+1/3(\gamma_{61}+\gamma_{71})x_1(t)x_2(t)^2+1/3(\gamma_{81}+\gamma_{91})x_1(t)^2x_2(t)+1/3\gamma_{10,1}x_2(t)^3\\
\mathbold{N}_2(\mathbold{x}(t))&=1/2\gamma_{12}x_1(t)^2+1/2(\gamma_{22}+\gamma_{32})x_1(t)x_2(t)+1/2\gamma_{42}x_2(t)^2+1/3\gamma_{52}x_1(t)^3\\&\quad+1/3(\gamma_{62}+\gamma_{72})x_1(t)x_2(t)^2+1/3(\gamma_{82}+\gamma_{92})x_1(t)^2x_2(t)+1/3\gamma_{10,2}x_2(t)^3,
 \end{split}
\end{equation*}
where $x_1(t)$ and $x_2(t)$ denote the displacements of the first and second modes, respectively, $\zeta_1$ and $\zeta_2$ denote the linear damping ratios, $\omega_1$ and $\omega_2$ denote the natural frequencies which here are assumed to be unknown, and $\gamma_{ij},i\in\{1,2,\ldots,10\},i\in\{1,2\},$ are the unknown parameters of the nonlinear term $\mathbold{N}(\mathbold{x}(t))$. It is worth mentioning that such a model can describe a wide range of mechanical structures with geometric nonlinearity, and is therefore not limited to the cross-beam example considered here.

For the cross-beam system, the unknown parameters are set as $\zeta_1=0.0076$, $\zeta_2=0.0026$, $\omega_1=101.6$, $\omega_2=104.6$, $\gamma_{11}=113.321$,  $\gamma_{21}=-104.755$, $\gamma_{31}=-104.755$, $\gamma_{41}=-29.740$, $\gamma_{51}=3.836\times 10^8$, $\gamma_{61}= 2.451\times 10^7$, $\gamma_{71}=4.902\times 10^7$, $\gamma_{81}=1.929\times 10^8$, $\gamma_{91}=9.644\times 10^7$, $\gamma_{10,1}=6.104\times 10^6$, $\gamma_{12}=-104.755$, $\gamma_{22}=-29.740$, $\gamma_{32}=-29.740$, $\gamma_{42}=85.367$, $\gamma_{52}=9.644\times 10^7$, $\gamma_{62}= 6.104\times  10^6$, $\gamma_{72}=1.221\times 10^7$, $\gamma_{82}=4.902\times 10^7$, $\gamma_{92}=2.451\times 10^7$, and $\gamma_{10,2}=2.351\times 10^6$. The frequency response of the structure for the first mode under the excitation $\mathbold{\sigma}(t)=[1.261\cos(\omega t)~ 0.318\cos(\omega t)]^\top$ is computed using numerical continuation and shown in Fig.~\ref{Ex2Bifurcation}. 

\begin{figure}[tb]
	\centering
	\includegraphics[width=0.49\textwidth]{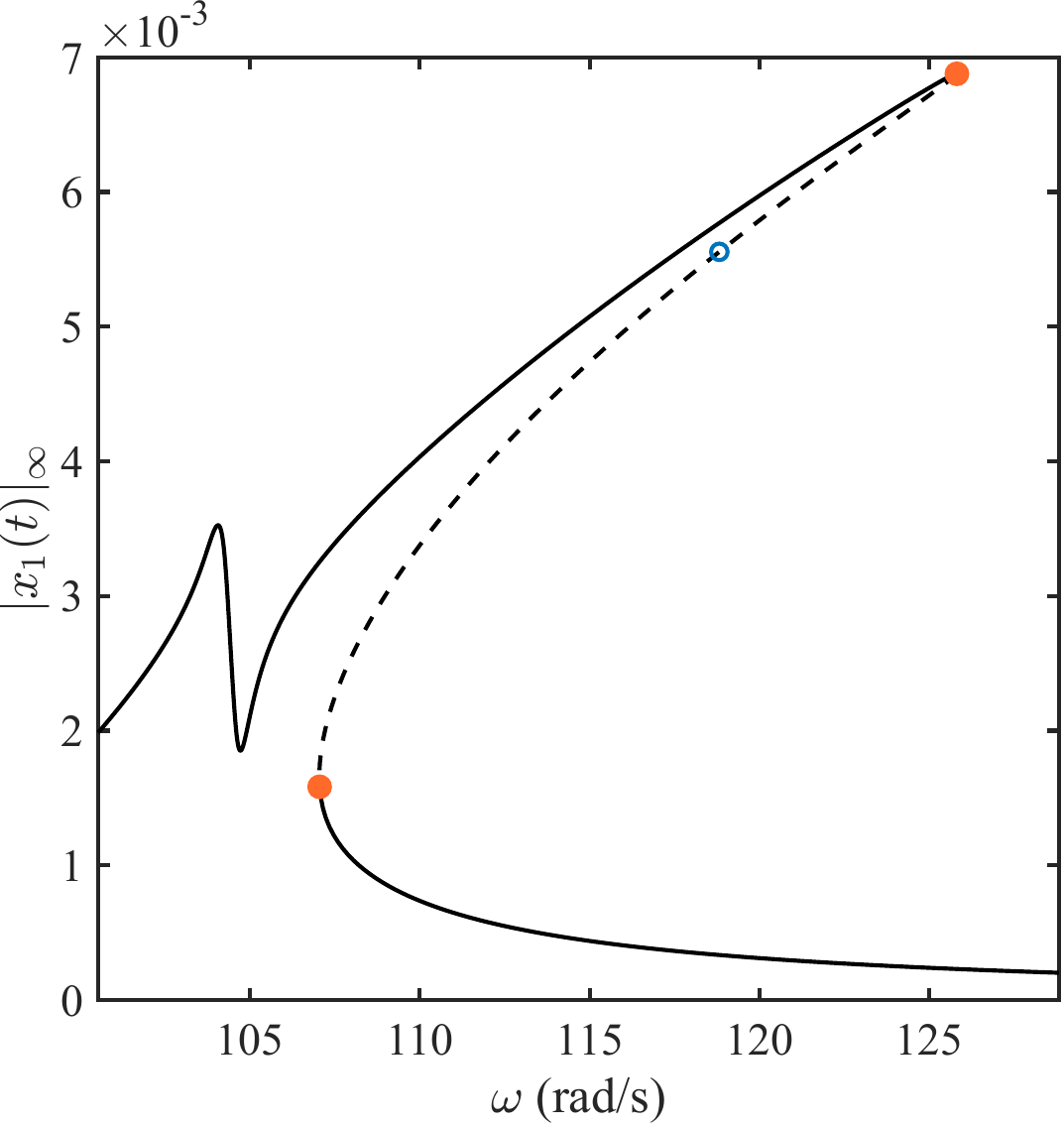}
	\caption{Frequency response curve associated with the first mode of the cross-beam structure. The specific point associated with $\mathbold{\xi}^*(t)$ is shown by a circle ($\textcolor{matlab_blue}{\circ}$), the dashed line represents unstable responses, and dots ($\textcolor{matlab_orange}{\bullet}$) show limit-point bifurcations.}
\label{Ex2Bifurcation}
\end{figure}

To demonstrate the control method, we consider a reference signal  $\mathbold{\xi}^r(t)=\mathbold{\xi}^*(t)=\begin{bmatrix}
	\dot{x}_1^*(t)  &  \dot{x}_2^*(t)&  x_1^{*}(t) & x_2^{*}(t)
\end{bmatrix}^\top$ which is an unstable response of the structure to $\mathbold{\sigma}(t)$ at $\omega=118.814$ (see the circle marker in Fig.~\ref{Ex2Bifurcation}). This desired response can be approximated as
\begin{equation*}
\begin{split}
      x_1^*(t)&\approx 10^{-4}\times  \big(- 35.344\cos(wt) + 0.521\cos(3wt)+ 0.002\cos(5wt)\\&\quad+ 42.08\sin(wt)+0.303\sin(3wt) - 0.006\sin(5wt)\big)\\
 x_2^*(t)&\approx 10^{-4}\times  \big(- 10.974\cos(wt)+ 0.132\cos(3wt)+ 0.001\cos(5wt)\\&\quad+ 12.358\sin(wt)+0.077\sin(3wt)- 0.002\sin(5wt)\big),\\
\end{split}
\end{equation*}
and hence the initial states associated with this response are $$(\dot{x}^*_1(0),\dot{x}^*_2(0),x^*_1(0),x^*_2(0))=(0.5104, 0.1495, -0.0035, -0.0011).$$

To achieve the desired response, we employ the proposed adaptive control strategy with $\mathbold{S}=2\times 10^4\mathbold{I}_{18}$, $k= 10^{-4}$, $\kappa=10^{-4}$, $\lambda_1=10^{-4}$, $\gamma=10^{3}$, and $\epsilon=1$. The obtained response and control inputs are presented in Fig.~\ref{Ex2-ClosedLoop}. According to the figures, the system response asymptotically converges to the desired response and the control input is noninvasive.

\begin{figure}[t]
\centering
\begin{tabular}{cc}
\subfloat[]{\includegraphics[width=0.45\textwidth]{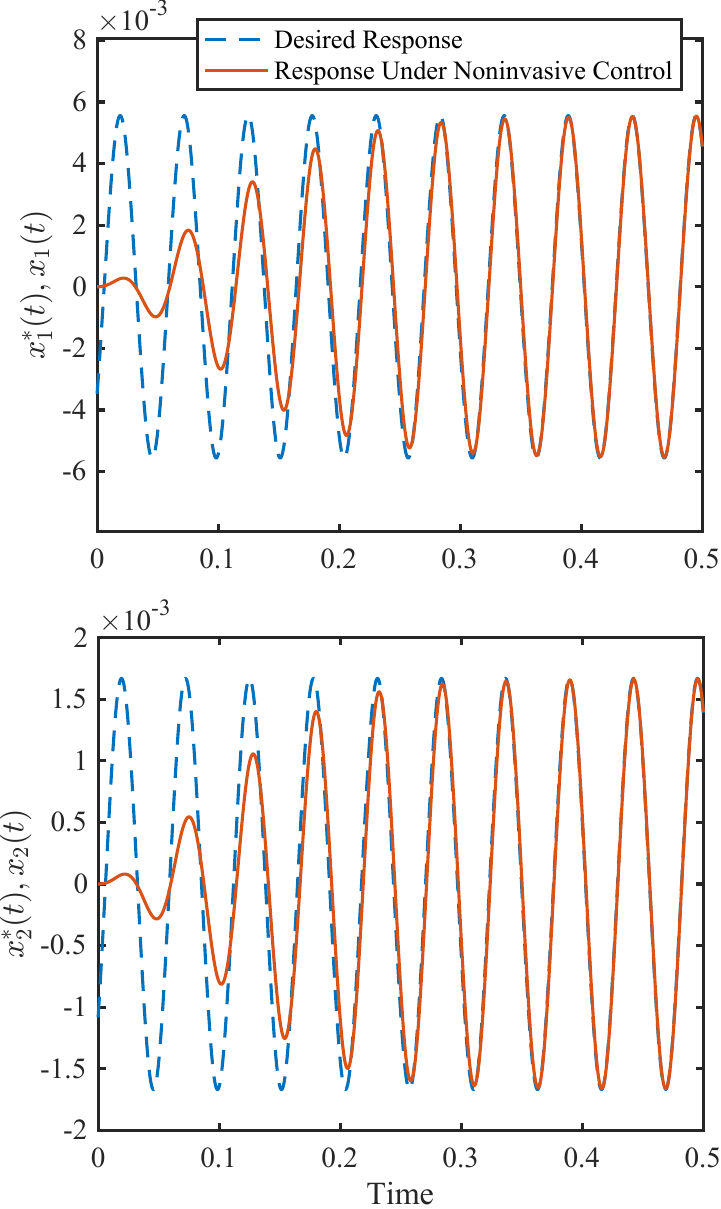}} & 
\subfloat[]{\includegraphics[width=0.45\textwidth]{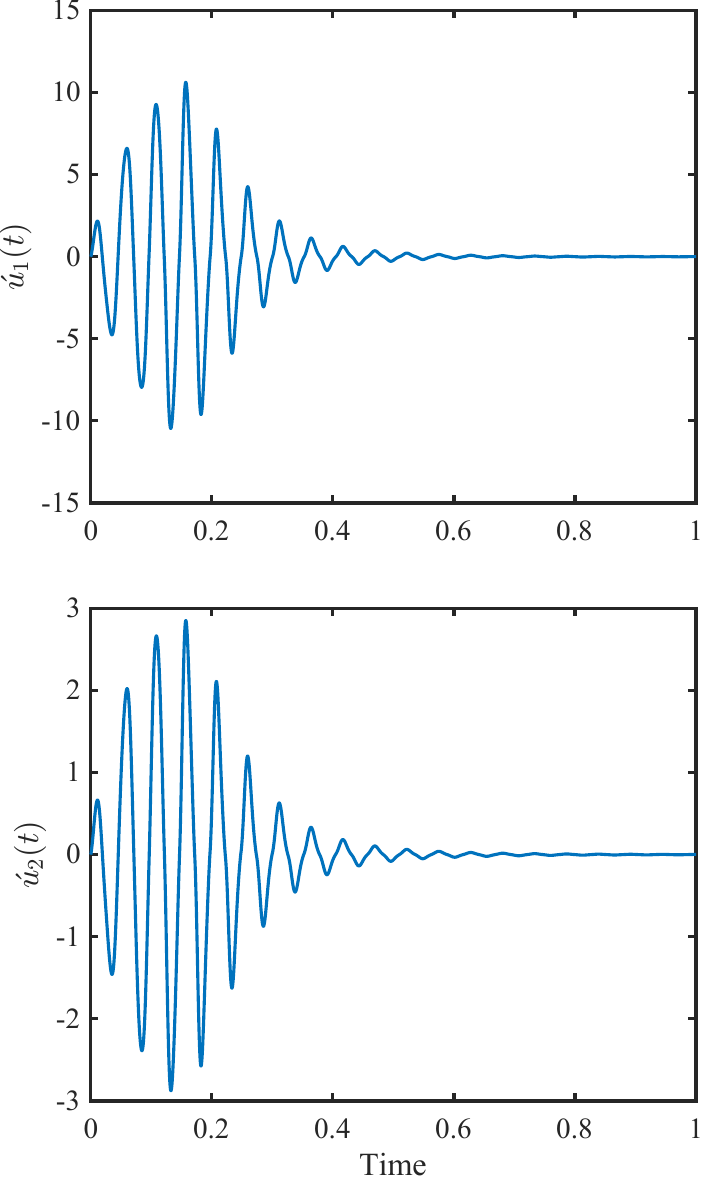}}
\end{tabular}
\caption{(a) Desired response of the cross-beam structure with the initial states $(\dot{x}_1(0),\dot{x}_2(0),x_1(0),x_2(0))=(0.5104,0.1495,-0.0035,-0.0011)$ and the
		controlled response of the cross-beam structure for the initial states $(\dot{x}_1(0),\dot{x}_2(0),x_1(0),x_2(0))=(0,0,0,0)$.
  (b) Noninvasive control inputs of the cross-beam structure.}
\label{Ex2-ClosedLoop}
\end{figure}

The proposed control strategy assumes the knowledge of the function $\mathbold{F}(\mathbold{\xi}(t))$. Simulations show that the proposed control strategy has some robustness in the presence of some uncertainties in $\mathbold{F}(\mathbold{\xi}(t))$. To show this, we consider the case where all the cross terms $x_1(t)x_2(t)$, $x_1(t)x_2(t)^2$, and $x_1(t)^2x_2(t)$ are overlooked in the controller whilst they are still present in the true system model. Simulation results show that the controller can still track the desired reference and achieve noninvasive control input (Fig. \ref{Ex22-ClosedLoop}). In this example, as the reference signal $\mathbold{\xi}^r(t)$ satisfies the true dynamical model of the uncontrolled system, by convergence of $\mathbold{\xi}(t)$ to $\mathbold{\xi}^r(t)$, the effect of the missed terms is compensated. This example illustrates that the tracking of a desired reference signal and the noninvasiveness of the control input can be obtained in the presence of model uncertainty provided the missed terms are not too large to be destabilizing. However, in general, it is always possible to find nonlinear terms that would lead to an unstable controlled system if not considered in the controller design (i.e., included in the true system model but not in the controller).

In the presence of model mismatch, and in the case when $\mathbold{\xi}^r(t)$ is not equal to a natural response of the system to the excitation $\mathbold{\sigma}(t)$, the system dynamics can either settle to a natural response of the system or not, regardless of the control reference. In the case where $\mathbold{\xi}(t)$ does not settle to a natural response $\mathbold{\xi}^*(t)$, the control input $\mathbold{u}'(t)$ will be invasive. To show this aspect of the proposed control strategy, we considered a simulation scenario where $\mathbold{\xi}^r(t)$ is chosen by considering up to 30 percent deviations in the Fourier coefficients describing the natural response $\mathbold{\xi}^*(t)$. The invasive control input of the system in this scenario is shown in Fig. \ref{Ex2up3}. In the case where $\mathbold{\xi}(t)$ settles to a natural response of the system $\mathbold{\xi}^*(t)$, the first part of the control input, $\mathbold{\eta}(t)$, will be compensated by the second part, $k(\mathbold{z}-\mathbold{x}^{r(n-1)}(0))$, such that $\mathbold{u}'(t)$ is noninvasive. Note that, in the context of CBC, the lack of tracking performance can present a difficulty even if the system settles to a natural response and the control input $\mathbold{u}'(t)$ is noninvasive, because the observed natural dynamics might be far from the reference $\mathbold{\xi}^r(t)$, leading to potentially significant jumps along the continuation path.

\begin{figure}[t]
\centering
\begin{tabular}{cc}
\subfloat[]{\includegraphics[width=0.45\textwidth]{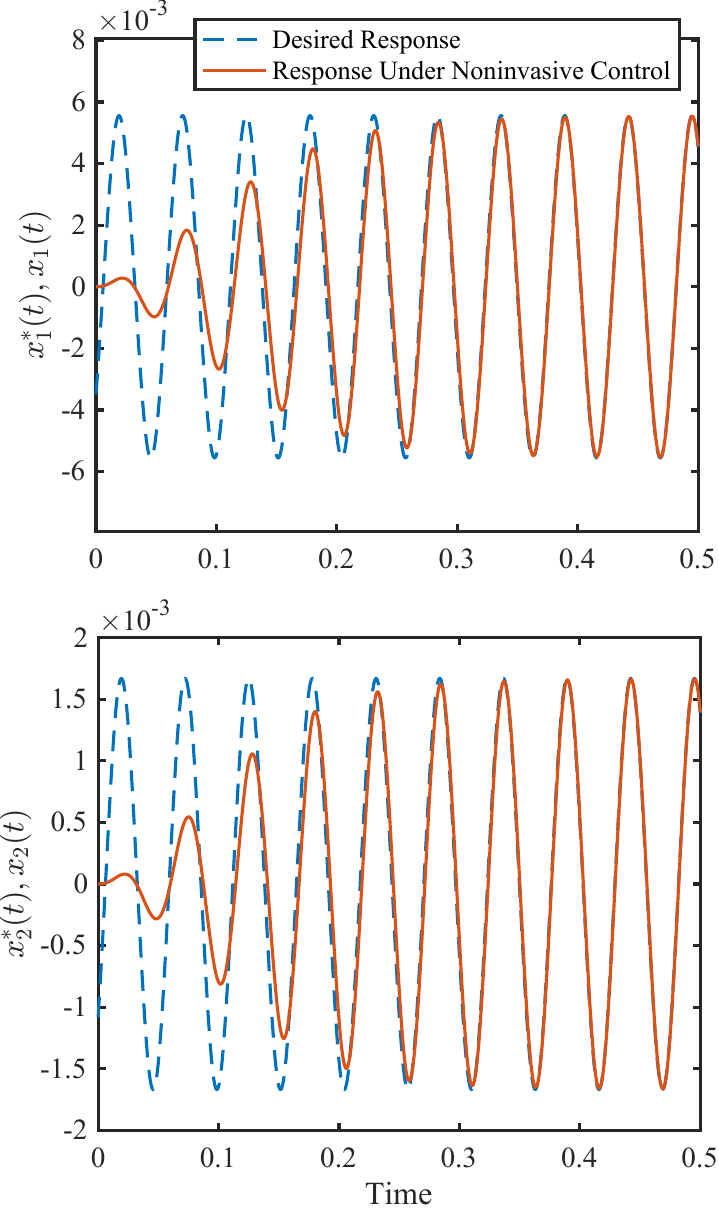}} & 
\subfloat[]{\includegraphics[width=0.45\textwidth]{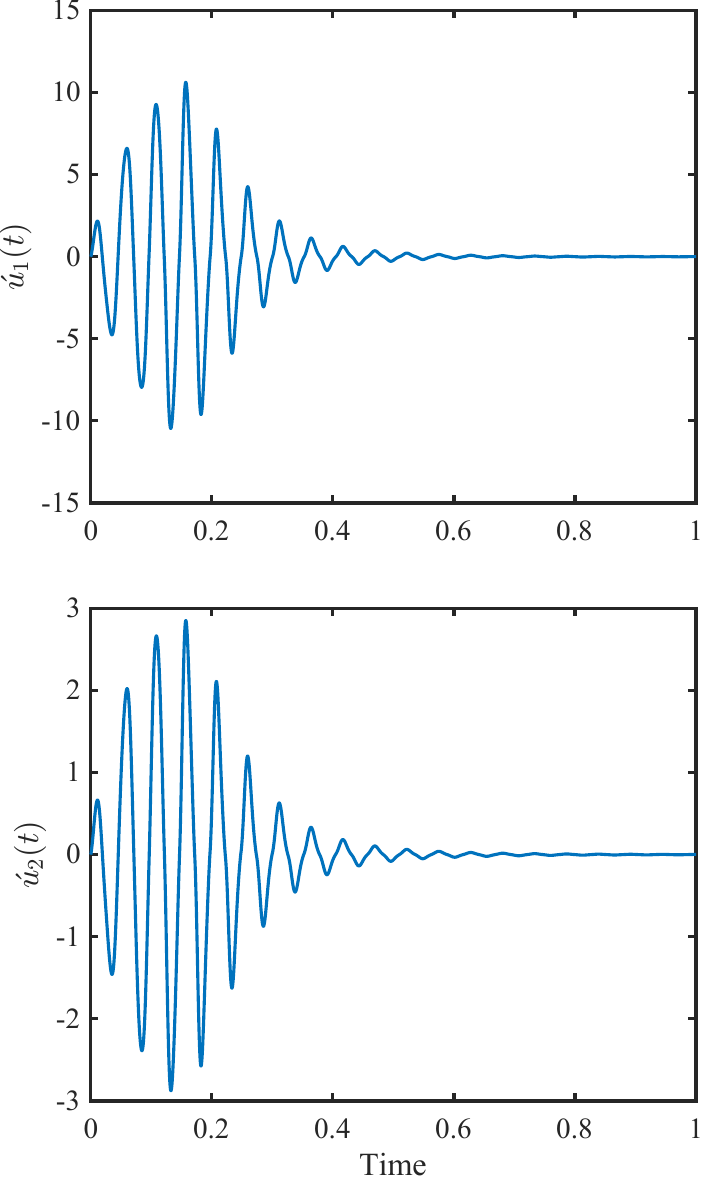}}
\end{tabular}
\caption{(a) Desired response of the cross-beam structure with the initial states $(\dot{x}_1(0),\dot{x}_2(0),x_1(0),x_2(0))=(0.5104,0.1495,-0.0035,-0.0011)$ and the
		controlled response of the cross-beam structure for the initial states $(\dot{x}_1(0),\dot{x}_2(0),x_1(0),x_2(0))=(0,0,0,0)$, when the model nonlinearities are partially unknown.
  (b) Noninvasive control inputs of the cross-beam structure, when the model nonlinearities are partially unknown.}
\label{Ex22-ClosedLoop}
\end{figure}

\begin{figure}[h]
	\centering	
	\includegraphics[width=0.95\textwidth]{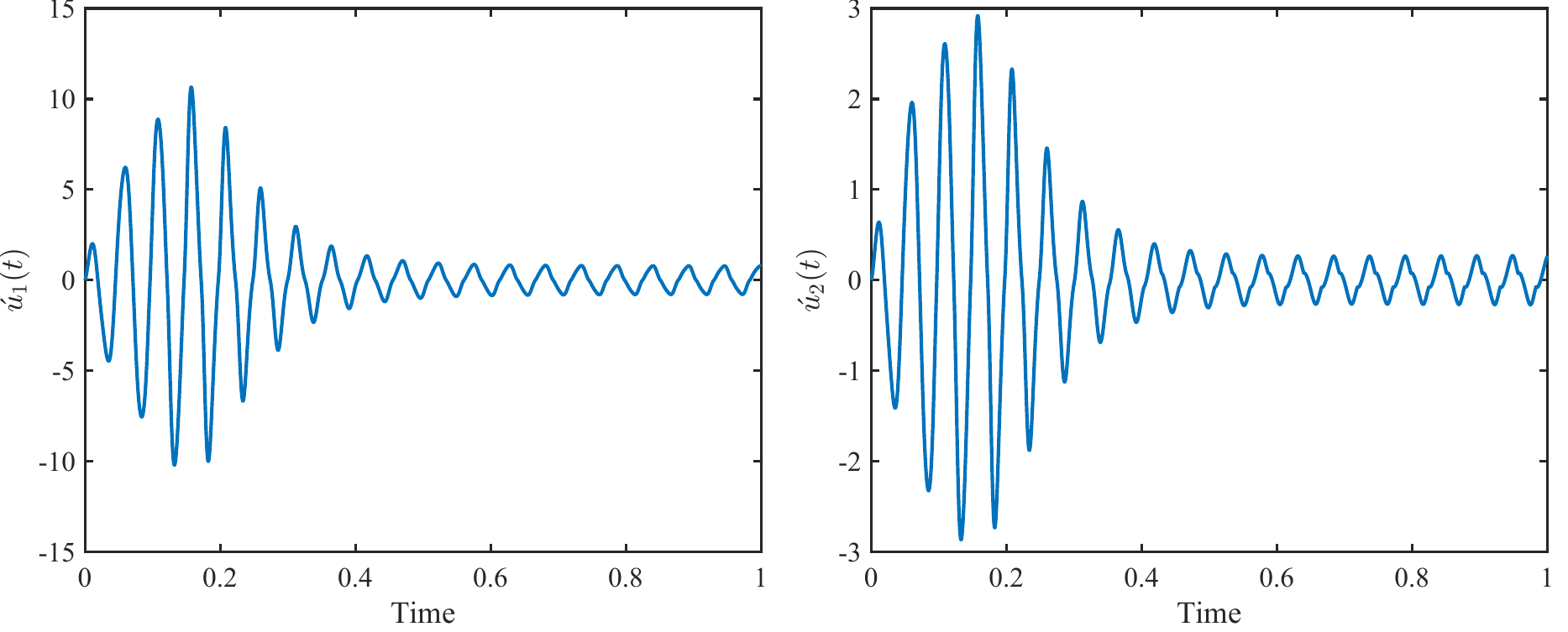}
		\caption{Invasive control input of the cross-beam structure when $\mathbold{\xi}^r(t)$ is not equal to a natural response of the system to the excitation $\mathbold{\sigma}(t)$ and the model nonlinearities are partially unknown.}
	\label{Ex2up3}
\end{figure}

\subsection{Cantilever beam with a nonlinear mechanism at its tip}
A cantilever beam structure with nonlinear springs attached at its free end is now considered. This structure has the particularity to exhibit a 1:3 mode interaction between its first and second bending modes~\cite{ShawMSSP:16}. The system can be described by the following modal equations:
\begin{equation*}
\ddot{\mathbold{x}}(t)+\mathbold{\Xi}\dot{\mathbold{x}}(t)+\mathbold{\Lambda} \mathbold{x}(t)+\mathbold{\Phi}^\top \mathbold{f}(\mathbold{\Phi}\mathbold{x}(t))=\mathbold{\Phi}^\top\mathbold{u}(t),
\end{equation*}
with
\begin{equation}\label{Swansea}
\begin{split}
\mathbold{x}(t)&=\begin{bmatrix}x_1(t)\\ x_2(t)\end{bmatrix},\mathbold{\Xi}=\begin{bmatrix}
                          2\zeta_1\omega_1 & 0 \\
   0 & 2\zeta_2\omega_2
         \end{bmatrix},\mathbold{\Lambda}=\begin{bmatrix}
 \omega_1^2 & 0 \\
0 & \omega_2^2
    \end{bmatrix}\\\mathbold{f}(\mathbold{\Phi}\mathbold{x}(t))&=\begin{bmatrix}0&0&0&2k_0\ell_0x'_4(t)\left(\frac{1}{a}-\frac{1}{\sqrt{a^2+{x'_4}(t)^2}}\right)\end{bmatrix}^\top,\mathbold{u}(t)=\begin{bmatrix}u_1(t)&u_2(t)&0&0\end{bmatrix}^\top\\
\mathbold{\Phi}\mathbold{x}(t)&=\begin{bmatrix}x'_1(t)&x'_2(t)&x'_3(t)&x'_4(t)\end{bmatrix}^\top=\mathbold{x}'(t),
 \end{split}
\end{equation}
where $x_1(t)$, $x_2(t)$, $\zeta_1$, $\zeta_2$, $\omega_1$, and  $\omega_2$ have the same meaning as in Section~\ref{Sec:CBS}, $k_0$ is the stiffness coefficient of the springs, $\ell_0$ is their original length, $a$ is the half span of the mechanism, and $\mathbold{\Phi}\in \mathbb{R}^{4\times 2}$ describes the relation between the modal variables and the physical deflection of the beam at four specific points along the beam denoted by $\mathbold{x}'(t)$. For this system, the parameters are $\zeta_1=0.01$, $\zeta_2=0.01$, $\omega_1=67.395$, $\omega_2=235.783$, $k_0=910$, $\ell_0= 0.018$, $a=0.019$, and 
\begin{equation*}
    \mathbold{\Phi}=\begin{bmatrix}
    -0.1603  & -0.6821\\
   -1.7748  & -4.4598 \\
   -5.9745   & 6.1940\\
   -6.1389   & 7.0245 
    \end{bmatrix}.
\end{equation*}
To describe~\eqref{Swansea} with the model class~\eqref{system}, $a$ and the mode shape matrix $\mathbold{\Phi}$ are assumed to be known, which in practice can be achieved using standard linear system identification techniques. 
All the other model parameters are considered unknown. Moreover, since the system has two degrees-of-freedom, we assume we have two control inputs. The response of the structure for the first mode under the excitation $\mathbold{\sigma}(t)=[2\cos(\omega t)~ 0 ~ 0~ 0]^\top$ is computed using numerical continuation and shown in Fig.~\ref{Ex3Bifurcation}.

\begin{figure}[h!]
	\centering
	\includegraphics[width=0.5\textwidth]{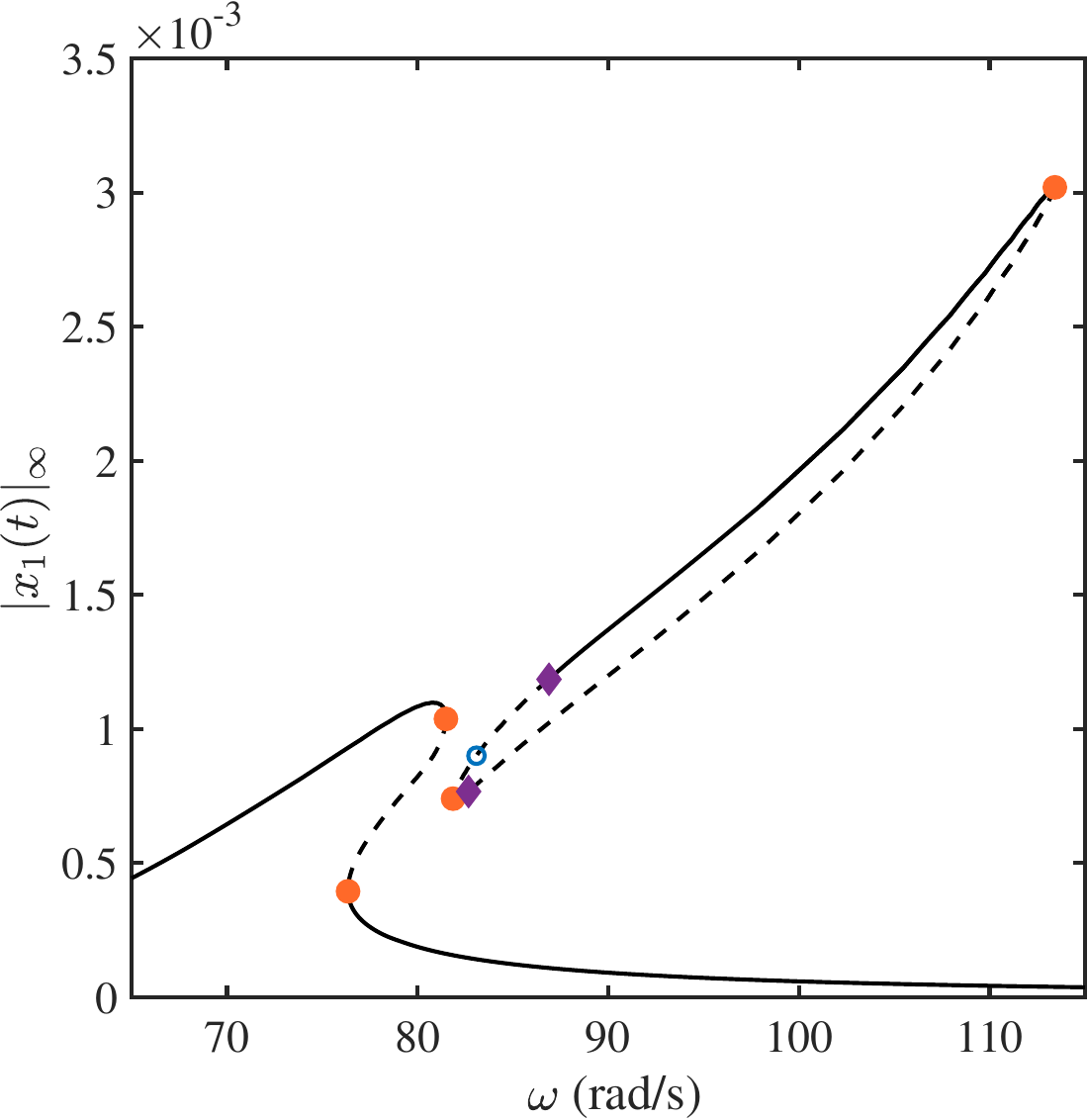}
	\caption{Frequency response associated with the first mode of the cantilever beam structure. The specific point associated with $\mathbold{\xi}^*(t)$ is shown by a circle ($\textcolor{matlab_blue}{\circ}$), the dashed line represents unstable responses, and dots ($\textcolor{matlab_orange}{\bullet}$) and diamonds ($\textcolor{matlab_purple}{\blacklozenge}$) show limit-point and Neimark-Sacker bifurcations, respectively.}
\label{Ex3Bifurcation}
\end{figure}

A reference signal
$\mathbold{\xi}^r(t)=\mathbold{\xi}^*(t)=\begin{bmatrix}
	\dot{x}_1^*(t)  &  \dot{x}_2^*(t)&  x_1^{*}(t) & x_2^{*}(t)
\end{bmatrix}^\top$ is considered to be an unstable response of the beam structure to $\mathbold{\sigma}(t)$ when $\omega= 83.085$ (see the circle marker in Fig.~\ref{Ex3Bifurcation}). 
This desired response can be approximated as
\begin{equation*}
\begin{split}
      x_1^*(t)&\approx 10^{-4}\times  \big(-2.834\cos(wt)+ 0.254\cos(3wt)-0.0341\cos(5wt)- 0.001\cos(7wt)\\
      &\quad- 8.241\sin(wt) + 0.066\sin(3wt)+ 0.026\sin(5wt)- 0.007\sin(7wt)\big) \\
 x_2^*(t)&\approx 10^{-4}\times  \big(- 0.487\cos(wt)- 2.6\cos(3wt)+0.055\cos(5wt)+ 0.001\cos(7wt)\\
      &\quad- 0.469\sin(wt)- 0.219\sin(3wt)- 0.044\sin(5wt)+ 0.009\sin(7wt)\big),
\end{split}
\end{equation*}
and hence the initial states associated with this response are 
$$(\dot{x}_1(0),\dot{x}_2(0),x_1(0),x_2(0))=(-0.066,-0.011,-2.613\times 10^{-4},-3.031\times 10^{-4}).$$

To achieve the desired natural response $\mathbold{\xi}^*(t)$, we employ the proposed adaptive control strategy with $\mathbold{S}=2\times 10^5\mathbold{I}_{6}$, $k=10^{-4}$, $\kappa=10^{-4}$, $\lambda_1=10^{-4}$, $\gamma=10^5$, and $\epsilon=1$. 
The simulation results, presented in Fig.~\ref{Ex3-ClosedLoop}-(a), show that the system response asymptotically converges to the desired one. As shown in Fig.~\ref{Ex3-ClosedLoop}-(b), the control input becomes noninvasive as the system reaches the desired response.

\begin{figure}[t]
\centering
\begin{tabular}{cc}
\subfloat[]{\includegraphics[width=0.45\textwidth]{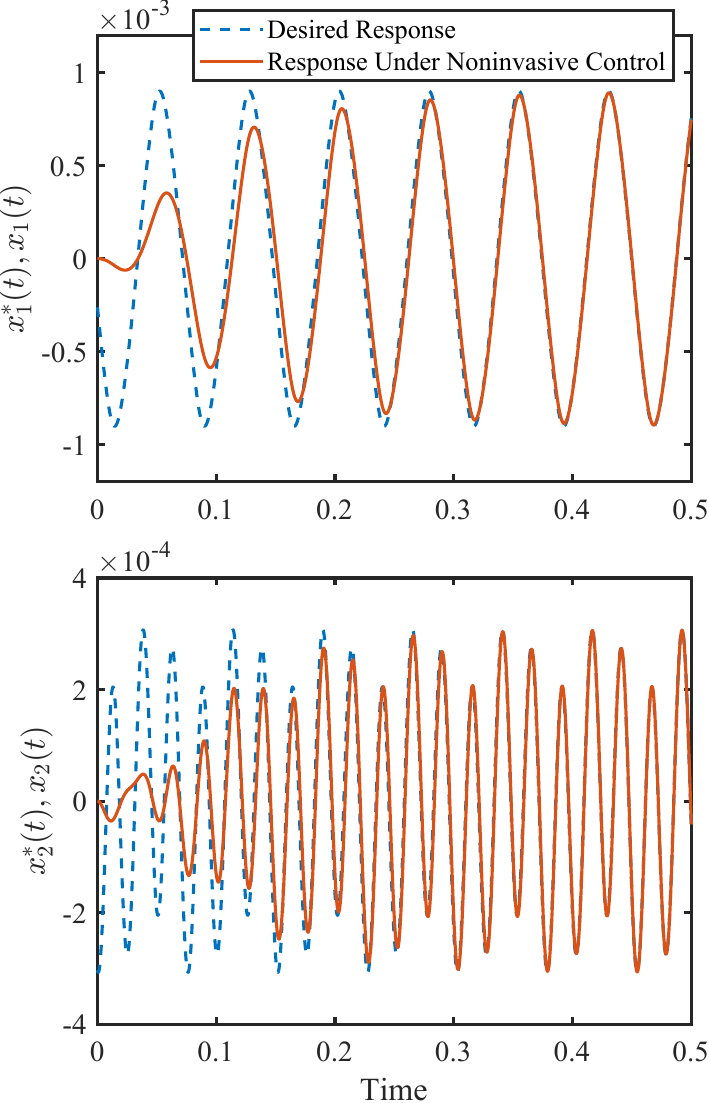}} & 
\subfloat[]{\includegraphics[width=0.45\textwidth]{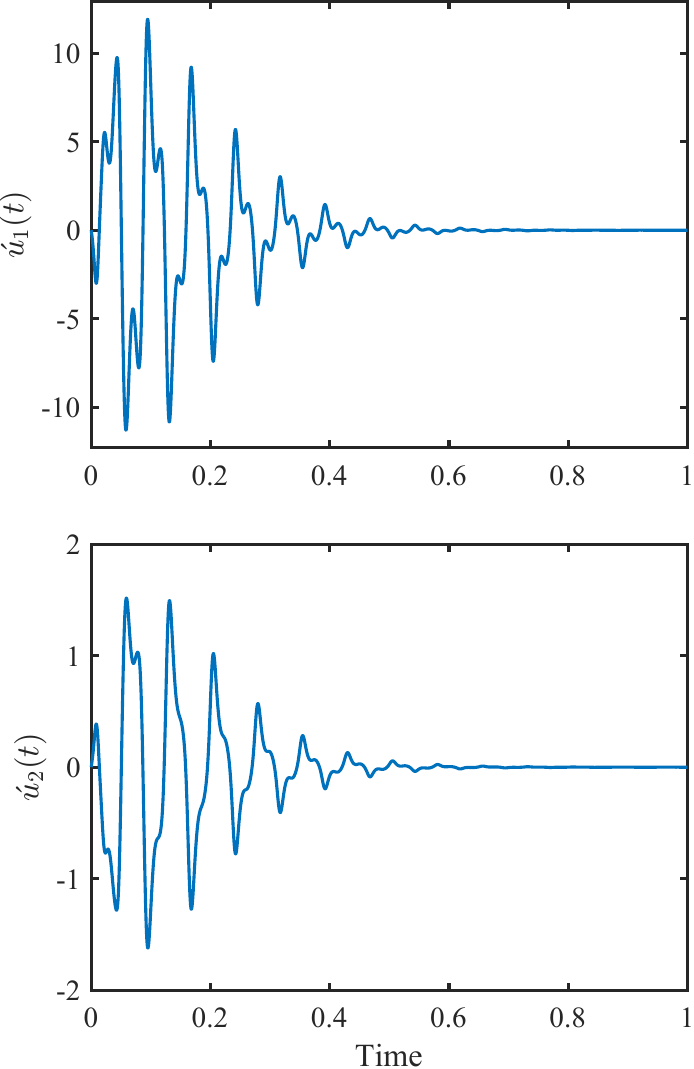}}
\end{tabular}
\caption{(a) Desired response of the cantilever beam structure  with the initial states $(\dot{x}_1(0),\dot{x}_2(0),x_1(0),x_2(0))=(-0.066,-0.011,-2.613\times 10^{-4},-3.031\times 10^{-4})$ and the controlled response of the cantilever beam structure for the initial states $(\dot{x}_1(0),\dot{x}_2(0),x_1(0),x_2(0))=(0,0,0,0)$. (b) Noninvasive control inputs of the cantilever beam structure.}
\label{Ex3-ClosedLoop}
\end{figure}

One of the contributions of the paper is that, contrary to the existing literature, the noninvasiveness of the proposed control strategy does not rely on the true estimation of the system parameters and thereby does not rely on the persistent excitation of the desired response. To show this, the norm of the estimation error $\widetilde{\mathbold{\theta}}(t)=\hat{\mathbold{\theta}}(t)-\mathbold{\theta}$ when $\hat{\mathbold{\theta}}(t)$ is initialized to zero and $\|\mathbold{\theta}\|=5.578\times 10^4$ is shown in Fig.~\ref{Ex3NonPE}-(a). According to the figure,  the estimation error does not converge to zero, and a negligible update in the model parameters $\hat{\mathbold{\theta}}(t)$ occurs. Defining the matrix
\begin{equation*}
\mathbold{M}_e(t)=\int_{t}^{t+s}\mathbold{F}^\top (\mathbold{\xi}(\tau))\mathbold{F}(\mathbold{\xi}(\tau)) d\tau,
\end{equation*}
the persistent excitation condition is satisfied if $\mathbold{M}_e(t)$, $t\rightarrow \infty$, is positive definite, i.e., its smallest eigenvalue is positive, for some real positive constant $s$  ($s$ can be the time period in the case of periodic response).
The smallest eigenvalue of $\mathbold{M}_e(t)$  along the system trajectory is shown in Fig. \ref{Ex3NonPE}-(b), which stays at zero and illustrates that $\mathbold{M}_e(t)$ is not positive definite. Hence, the controlled system is not persistently excited along the system trajectory.

\begin{figure}[t]
\centering
\begin{tabular}{cc}
\subfloat[]{\includegraphics[width=0.45\textwidth]{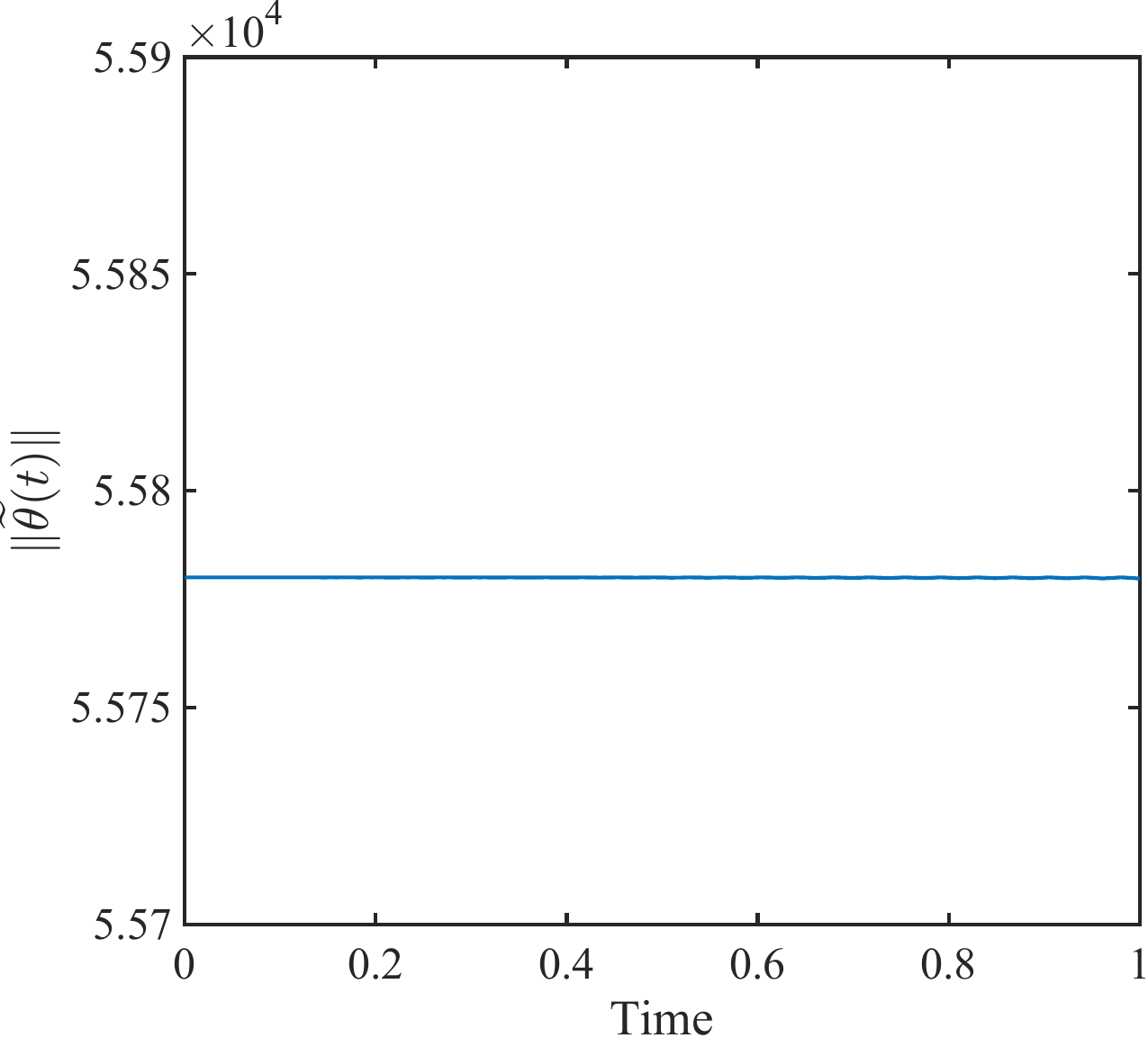}} & 
\subfloat[]{\includegraphics[width=0.45\textwidth]{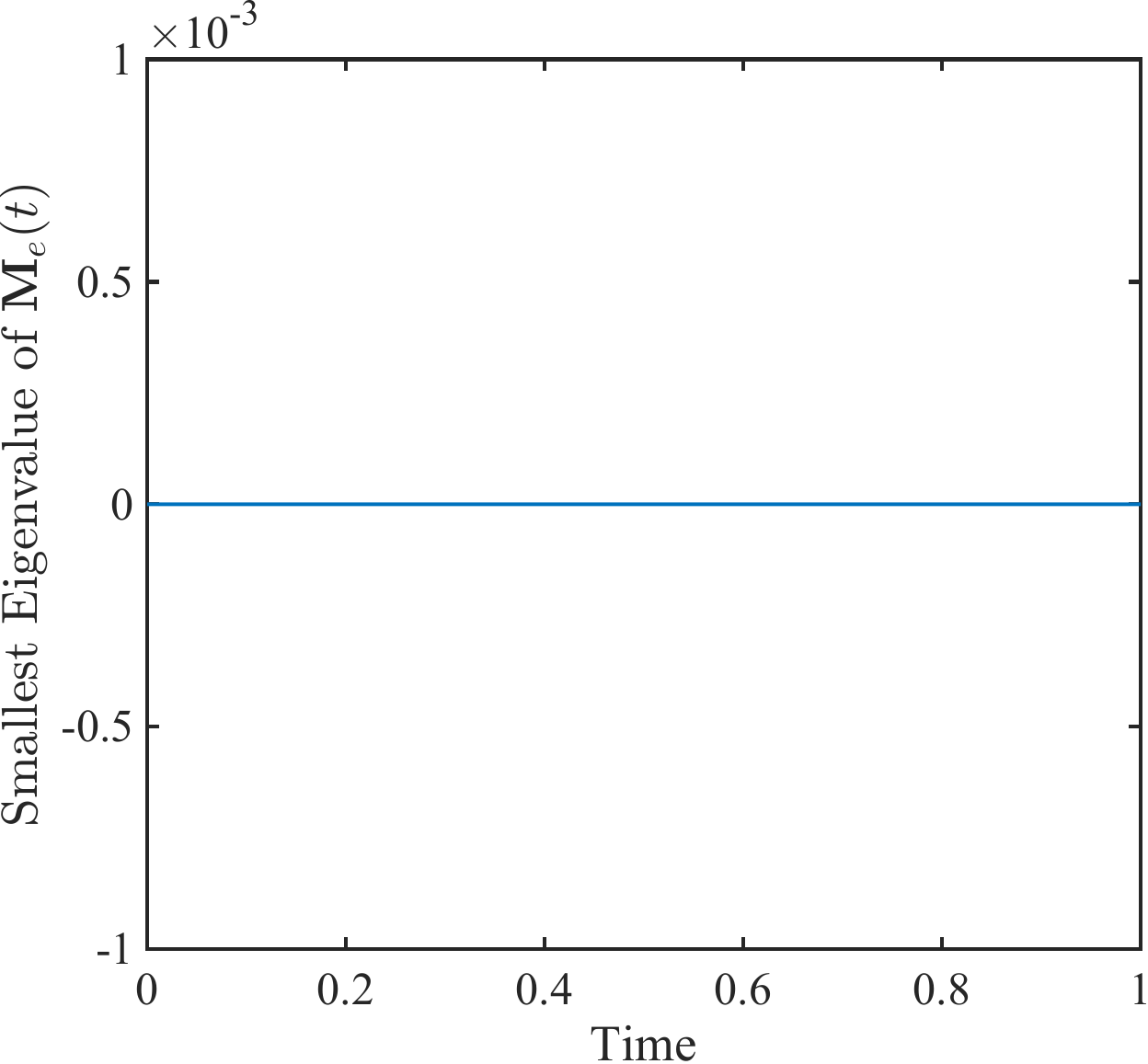}}
\end{tabular}
\caption{(a) Error of estimating the parameters of the cantilever beam structure. (b) The smallest eigenvalue of the matrix $\mathbold{M}_e(t)$ associated with the cantilever beam structure.}
\label{Ex3NonPE}
\end{figure}

\section{Conclusion and Future Work}~\label{SecConclusions}
Control-based continuation is an increasingly popular method to characterize the nonlinear dynamics of a physical system directly during tests, without the need for a model. Despite being successfully applied to a wide range of experiments and providing invaluable results for model calibration and validation, the general and systematic design of the noninvasive controller at the core of the method is lacking.

In this paper, an adaptive strategy for noninvasive control of nonlinear systems that are linear in their unknown parameters is proposed. The method guarantees the convergence to the desired response and the noninvasiveness of the control input if and only if the reference signal  is a natural response of the uncontrolled system to  the excitation  $\mathbold{\sigma}(t)$. Compared to other available methods, the proposed control strategy does not require any knowledge of the system parameters nor their exact estimation, and it does not assume stable linearized dynamics. This makes the proposed methodology applicable to a much wider range of nonlinear systems, with a wider range of response types. Furthermore, the proposed method does not rely on the persistence of the excitation, which, in the context of CBC, means there is no need to guarantee/check the persistent nature of the excitation during experiments for the noninvasiveness of the controller. This has the potential to increase CBC accuracy and reduce overall testing time.

\end{document}